\documentclass[12pt]{preprint}
\usepackage[a4paper,innermargin=1.3in,outermargin=1.7in,bottom=1.5in,marginparwidth=1.5in,marginparsep=3mm]{geometry}

\usepackage[full]{textcomp}
\usepackage[osf]{newtxtext}
\usepackage{cabin}
\usepackage{mathalfa}
\usepackage{microtype}

\usepackage{amsmath}
\usepackage{amsfonts}
\usepackage{amssymb}
\usepackage{hyperref}
\usepackage{mathrsfs}
\usepackage{mhequ}
\usepackage{mhenvs}
\usepackage{Symbols}
\usepackage[utf8]{inputenc} 
\usepackage{tikz}
\usepackage{breakurl}
\usepackage{skull}
\usepackage{wasysym}
\usepackage{centernot}

\def\emptyset{\mathop{\centernot\ocircle}}

\newcommand{\E}{\mathbf E}
\newcommand{\obar}{\overline}
\newcommand{\bU}{\overline U}

\newcommand{\death}{\raisebox{-0.17em}{\includegraphics[width=3.3mm]{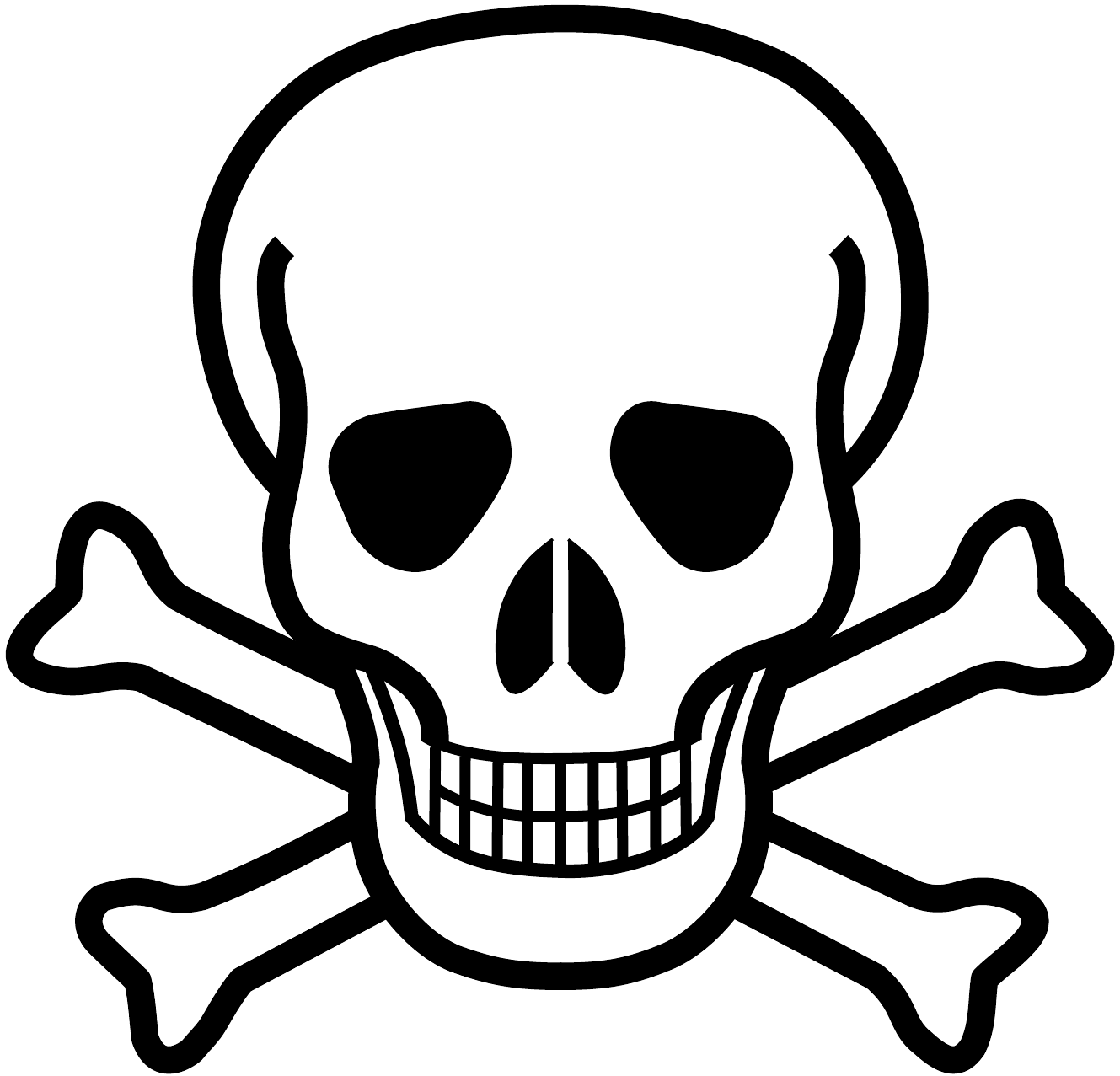}}}

\renewcommand{\P}{\mathbf P}
\newcommand{\MM}{\mathscr M}
\newcommand{\LL}{\mathscr L}
\newcommand{\RR}{\mathfrak R}
\newcommand{\K}{\mathfrak K}

\definecolor{darkred}{rgb}{0.7,0.1,0.1}
\definecolor{darkblue}{rgb}{0.1,0.1,0.8}

\def\s{\mathfrak s}

\def\ad{\mathrm{pr}}
\def\restr{\mathord{\upharpoonright}}
\def\T{\mathbf{T}}
\def\PPi{\boldsymbol{\Pi}}
\def\PPhi{\boldsymbol{\Phi}}
\def\hPPhi{\hat{\boldsymbol{\Phi}}{}}
\def\mod{\boldsymbol{\xi}}
\def\PP{\mathfrak{P}}
\def\restr{\mathord{\upharpoonright}}
\def\V{E}
\def\${|\!|\!|}
\def\deg{\mathop{\mathrm{deg}}}
\def\hdeg{\mathop{\overline{\mathrm{deg}}}}
\def\noise{f}

\newtheorem{assumption}{Assumption}

\usetikzlibrary{shapes.misc}
\usetikzlibrary{shapes.symbols}
\tikzset{
	dot/.style={circle,fill=black,draw=black, solid,inner sep=0pt,minimum size=0.5mm},
	doth/.style={circle,fill=white,draw=black, solid,inner sep=0pt,minimum size=0.5mm},
	yy/.style={circle,fill=gray!20,draw=black,inner sep=0pt,minimum size=0.8mm},
	>=stealth,
	}
\makeatletter
\def\DeclareSymbol#1#2#3{\expandafter\gdef\csname MH@symb@#1\endcsname{\tikz[baseline=#2,scale=0.15]{#3}}%
\expandafter\gdef\csname MH@symb@#1s\endcsname{\scalebox{0.6}{\tikz[baseline=#2,scale=0.15]{#3}}}}
\def\<#1>{\csname MH@symb@#1\endcsname}
\makeatother

\DeclareSymbol{X}{-2.4}{\node[dot] {};}
\DeclareSymbol{1}{0}{\draw[white] (-.4,0) -- (.4,0); \draw (0,0)  -- (0,1.2) node[dot] {};}
\DeclareSymbol{1h}{0}{\draw[white] (-.4,0) -- (.4,0); \draw (0,0)  -- (0,1.2) node[doth] {};}
\DeclareSymbol{2}{0}{\draw (-0.5,1.2) node[dot] {} -- (0,0) -- (0.5,1.2) node[dot] {};}
\DeclareSymbol{2h}{0}{\draw (-0.5,1.2) node[dot] {} -- (0,0) -- (0.5,1.2) node[doth] {};}
\DeclareSymbol{3}{0}{\draw (0,0) -- (0,1.2) node[dot] {}; \draw (-.7,1) node[dot] {} -- (0,0) -- (.7,1) node[dot] {};}
\DeclareSymbol{3h}{0}{\draw (0,0) -- (0,1.2) node[dot] {}; \draw (-.7,1) node[dot] {} -- (0,0) -- (.7,1) node[doth] {};}
\DeclareSymbol{31}{-3}{\draw (0,0) -- (0,-1) -- (1,0) node[dot] {}; \draw (0,0) -- (0,1.2) node[dot] {}; \draw (-.7,1) node[dot] {} -- (0,0) -- (.7,1) node[dot] {};}
\DeclareSymbol{30}{-3}{\draw (0,0) -- (0,-1); \draw (0,0) -- (0,1.2) node[dot] {}; \draw (-.7,1) node[dot] {} -- (0,0) -- (.7,1) node[dot] {};}
\DeclareSymbol{32}{-3}{\draw (0,0) -- (0,-1) -- (1,0) node[dot] {}; \draw (0,0) -- (0,-1) -- (-1,0) node[dot] {}; \draw (0,0) -- (0,1.2) node[dot] {}; \draw (-.7,1) node[dot] {} -- (0,0) -- (.7,1) node[dot] {};}
\DeclareSymbol{12}{-3}{\draw (0,0) -- (0,-1) -- (1,0) node[dot] {}; \draw (0,0) -- (0,-1) -- (-1,0) node[dot] {}; \draw (-.7,1) node[dot] {} -- (0,0);}
\DeclareSymbol{10}{-3}{\draw (-.7,1) node[dot] {} -- (0,0) -- (0,-1);}

\DeclareSymbol{32h}{-3}{\draw (0,0) -- (0,-1) -- (1,0) node[dot] {}; \draw (0,0) -- (0,-1) -- (-1,0) node[dot] {}; \draw (0,0) -- (0,1.2) node[dot] {}; \draw (-.7,1) node[doth] {} -- (0,0) -- (.7,1) node[dot] {};}
\DeclareSymbol{30h}{-3}{\draw (0,0) -- (0,-1); \draw (0,0) -- (0,1.2) node[dot] {}; \draw (-.7,1) node[doth] {} -- (0,0) -- (.7,1) node[dot] {};}
\DeclareSymbol{12h}{-3}{\draw (0,0) -- (0,-1) -- (1,0) node[dot] {}; \draw (0,0) -- (0,-1) -- (-1,0) node[dot] {}; \draw (-.7,1) node[doth] {} -- (0,0);}
\DeclareSymbol{10h}{-3}{\draw (-.7,1) node[doth] {} -- (0,0) -- (0,-1);}

\DeclareSymbol{22}{-3}{\draw (0,0.3) -- (0,-1) -- (1,0) node[dot] {}; \draw (0,0.3) -- (0,-1) -- (-1,0) node[dot] {};\draw (-.7,1) node[dot] {} -- (0,0.3) -- (.7,1) node[dot] {};}
\DeclareSymbol{20}{-3}{\draw (0,0) -- (0,-1);\draw (-.7,1) node[dot] {} -- (0,0) -- (.7,1) node[dot] {};}
\DeclareSymbol{32W}{-3}{\draw  [densely dotted, semithick] (-1.5,0)node[dot] {}  -- (0,-1.5) node[yy]{} -- (1.5,0) node[dot] {}; \draw (.8,-2.3)node{\tiny $y$} ;\draw[semithick] (0,0) -- (0,-1.5) node[yy] {}; \draw [densely dotted, semithick] (0,0) -- (0,1.8) node[dot] {}; \draw[densely dotted, semithick] (-1,1.5) node[dot] {} -- (0,0) -- (1,1.5) node[dot] {};}
\DeclareSymbol{32WW}{-3}{\draw  [densely dotted, semithick] (-1.5,0)node[dot] {}  -- (0,-1.5) node[yy]{} -- (1.5,0) node[dot] {}; \draw (.8,-2.3)node{\tiny $\bar{y}$} ;\draw[semithick] (0,0) -- (0,-1.5) node[yy] {}; \draw [densely dotted, semithick] (0,0) -- (0,1.8) node[dot] {}; \draw[densely dotted, semithick] (-1,1.5) node[dot] {} -- (0,0) -- (1,1.5) node[dot] {};}

\begin{document}

\title{The strong Feller property for\\ singular stochastic PDEs}
\author{M.~Hairer$^1$ and J.~Mattingly$^2$}
\institute{Mathematics Research Centre, University of Warwick
 \\ \email{m.hairer@warwick.ac.uk}
\and Departments of Mathematics and Statistical Science, Duke University\\
 \email{jonm@math.duke.edu}}
\titleindent=0.65cm

\maketitle
\thispagestyle{empty}

\begin{abstract}
We show that the Markov semigroups generated by a large class of singular stochastic 
PDEs satisfy the strong Feller property. These include for example the KPZ equation
and the dynamical $\Phi^4_3$ model.
As a corollary, we prove that the Brownian bridge measure is the unique 
invariant measure for the KPZ equation with periodic boundary conditions.
\end{abstract}

\tableofcontents

\section{Introduction}

Recall that a Markov operator $\PP$ on some separable metric space $\CX$
is said to satisfy the \textit{strong Feller property} if it maps
all bounded measurable functions $\CX \to \R$ into bounded continuous functions. 
Intuitively, the strong Feller property allows us to link measure theoretic
properties of a process to the corresponding topological properties. As a result, it
is a very useful ingredient when trying to establish the ergodicity of a 
given Markov process, see for example \cite{MeynTweedie}.
When considering finite-dimensional diffusions with sufficiently smooth coefficients, 
the strong Feller property is a 
consequence of the parabolic H\"ormander condition \cite{Hormander,Malliavin,Norris}, 
and this condition is essentially sharp. 

In infinite dimensions, no such sharp condition exists, despite considerable progress
on notions of hypoellipticity in that case,
see for example \cite{SFDegenerate,Josef,MR2257860,Annals,EJP}. This should of course
not come as a surprise since measures on infinite-dimensional spaces have a rather
annoying tendency of being mutually singular. The strong Feller property on the other hand
implies that the total variation distance between transition probabilities
starting from nearby points is small \cite{Sei02,NotesNE}, thus ruling out mutual singularity.
However when considering parabolic stochastic PDEs in bounded domains driven by noise that
is ``sufficiently non-degenerate and not overly smooth spatially,'' it is well-known that the strong Feller property holds,
see for example \cite{SFOld1,SFOld2,SFOld3,DPZ,SFDegenerate}. 

All these results do however rely very strongly on the well-posedness
of the equations under consideration, as well as on good a priori
control on their derivative with respect to initial conditions. This
typically enforces some conditions on the driving noise requiring it
to be sufficiently regular for the standard solution theory
\cite{ref2} to apply.  The aim of this article is to extend these
results to singular stochastic PDEs like the dynamical $P(\Phi)_2$
model \cite{Jona,DPD2,AlbRock91}, the KPZ equation \cite{KP,KPZ}, and
the dynamical $\Phi^4_3$ model \cite{reg}. 

Both of these examples contain noise which is spatially rough enough to
make the existence and uniqueness theory extremely
non-trivial. However, when
proving  the strong Feller property, rougher noise only makes the  core
of the
proof easier, the only difficulties being making sense of all of the
objects being manipulated. This is in contrast to the case of
spatially smooth or degenerate noise where the existence and uniqueness of solutions and the
correctness of all manipulations is straightforward while the  proof
of any property like the strong Feller property which links the long
time behaviour of nearby points and illuminates the ergodic properties
of the system are more complicated, see for example \cite{Annals,EJP}.

The main novelty of our
approach is that while it is close in spirit to proofs based on the
Bismut-Elworthy-Li formula \cite{EL,SFOld3}, we do not require any a
priori control on the solutions: they are in principle allowed to blow
up, even with positive probability. In particular, this allows to
strengthen well-posedness results for almost every initial condition
as in \cite{AlbRock91,DPD2,Konstantin} to every initial condition in the
topological support of the invariant measure.  It also yields as a
corollary the uniqueness of the Brownian bridge measure (modulo height
shifts) for the KPZ equation with periodic boundary conditions, which
had not been established before.

Both our approach and that used in the proof based on the Bismut-Elworthy-Li formula
rely on transferring the variation caused by shifting the
initial condition infinitesimally to an infinitesimal shift in the
noise. In versions of the argument closest
to ours, one then integrates by parts against the Gaussian measure,
moving the infinitesimal variation in the noise to the Wiener measure
and finally averaging over the realisations of the noise. This in particular requires the
solution to be well posed almost surely. In our approach we accumulate these
infinitesimal shifts to build a macroscopic shift in
the noise corresponding to a macroscopic shift in the initial
condition. This allows us to work in a more pathwise manner and
consider equations which might explode with positive probability.

\begin{remark}
  It is very natural to ask whether the solutions to the stochastic
  Navier-Stokes equations driven by space-time white noise in
  dimensions $2$ and $3$ as constructed in \cite{DPD,Zhu} also satisfy
  the strong Feller property. While we believe this to be true, our
  results as they are presented in Section~\ref{sec:reg} do not cover
  this case because of the presence of the Leray projection. We do
  however expect the results of Section~\ref{sec:SF} to be applicable
  to that situation as well by slightly modifying the argument of
  Section~\ref{sec:reg}.  We feel that this is largely a technical
  issue but do not explore it here.
\end{remark}

Let us also mention that some related results have recently been obtained. In \cite{RZZ},
the authors show that the $\Phi^4_2$ measure (as constructed in \cite{Nelson})
 is indeed the unique invariant
measure for the dynamical $\Phi^4_2$ model. The proof of this fact however does not 
make use of the strong Feller property but instead relies on an asymptotic coupling
argument. While this is sufficient to prove many ergodic properties,
it does not establish the local regularity in the total variation
topology of the transition density. More recently, in \cite{Pavlos}, the authors obtained not only the strong Feller property
for the dynamical $P(\Phi)_2$ model, but also the exponential ergodicity of
the dynamical $\Phi^4_2$ model.
While these results are much stronger than ours, they are restricted to one particular model
and rely strongly on good a priori bounds on the solutions which are not
available in all the cases we treat.

The structure of this article goes as follows. In Sections~\ref{sec:setup} and \ref{sec:SF},
we set up an abstract framework and give sufficient conditions for a Markovian 
continuous random dynamical system to satisfy the strong Feller property. 
This framework is very general and we expect it to be useful also in other contexts.
It is
designed so that, as shown in Section~\ref{sec:reg}, it covers a very large class
of semilinear stochastic PDEs, provided that we build their solutions via 
the theory of regularity structures.
We do however expect that constructions using paracontrolled calculus
as in \cite{Paracontrol,Khalil,Reloaded} can also be fitted into our framework.
In our last section, we then finally show that many interesting examples of singular 
stochastic PDEs are covered by our results.

\subsection*{Acknowledgements}

\small{This work was initiated during the programme ``New Challenges
  in PDE: Deterministic Dynamics and Randomness in High and Infinite
  Dimensional Systems'' which both authors attended at MSRI Berkeley.
  MH gratefully acknowledges financial support from the Philip
  Leverhulme Trust and from the European Research Council. JCM
  gratefully acknowledges financial support of the Simons Fondation
  through a collaboration grant
  and of the NSF though the grants
  DMS-1613337 and DMS-1546130.}

{\small }

\section{Abstract setup}
\label{sec:setup}
In this section, we set up the abstract framework needed to make
precise the idea that an infinitesimal variation in the initial condition can
be equated with an infinitesimal variation in the noise. This is the
content of equation \eqref{e:propertyShift} in
Assumption~\ref{ass:shiftmap} below. It is in establishing
\eqref{e:propertyShift} for an adapted shift  that we  require
noise which is non-degenerate and with sufficient spatial roughness, see also Assumption~\ref{ass:nondegenerate} below.

Over the next two sections, we will consider a general random dynamic in some
Banach space $U$, defined over
a probability space
$(\Omega,\mathcal{F},\P)$. We will always consider Gaussian probability spaces
endowed with a filtration. More precisely, we assume that 
$\Omega$ is a separable Banach space and there exists a Hilbert space $H_0$ such that
$\CH = L^2(\R,H_0) \subset \Omega$ is the Cameron-Martin space for the 
Gaussian measure $\P$.
The canonical random variable $\omega$
 drawn from $\Omega$ according to $\P$ induces the two-sided continuous filtration
$\{\mathcal{F}_{s,t}, s<t\}$ by $\mathcal{F}_{s,t} = \sigma\{
h^*(\omega)\,:\, h \in L^2([s,t],H_0)\}$, with the canonical inclusion
$L^2([s,t],H_0) \subset \CH$, where we use the canonical identification
$h \leftrightarrow h^*$ between the Cameron-Martin space and measurable linear
functionals on $\Omega$, see \cite{Bogachev,SPDEnotes}. We say that a stochastic process over
$(\Omega,\mathcal{F},\P)$ is adapted if it is adapted to the filtration
$\CF_t \eqdef \CF_{-\infty,t}$.

We also assume that we are given a complete separable metric space $\CM$ as well as a
random variable
\begin{equ}
  \mod\colon \Omega \rightarrow \CM\;.
\end{equ}
In the applications we have in mind, $\CM$ is the space of admissible
models for a given regularity structure as in \cite{reg} and $\mod$ is the map 
constructing a suitable admissible model from the underlying 
driving noise. We will generally write $\mod$ for the
$\CM$-valued random variable given by $\mod=\mod(\omega)$ where
$\omega$ is drawn according to $\P$ and $\xi$ for a generic element of $\CM$.
For any measurable map $X\colon \Omega \to \CX$ with $\CX$ a Polish space,
and for every $h \in \CH$, we furthermore use without further ado the
notation $\omega \mapsto X(\omega+h)$ as a shortcut for the (unique up to
null sets) random variable $Y$
such that $\E F(Y) = \E F(X) \exp(h^*(\omega) - \|h\|_\CH^2/2)$ for every
bounded measurable $F \colon \CX \to \R$.

The type of random dynamics that we are interested in are then defined 
on a Banach space $U$ and described by a family of maps
\begin{equ}
   \Phi_{s,t} \colon\bU \times \CM \rightarrow \bU\;,\qquad
   0 \le s \le t \le 1\;,
\end{equ}
where we have extended the state space $U$ to include a ``death
state'' by defining 
$\bU= U \cup \{\death\}$. The space $\bU$ is again a separable metric space
by setting for example $d(\death,u)^2 = 1 + \|u\|^2$ for all $u \in U$.
These maps are assumed to be consistent in the sense that,
for any $t \in [s,r]$ with $s \le r$, one has the identity
\begin{equ}[e:consistency]
\Phi_{t,r}(\Phi_{s,t}(u,\xi),\xi) = \Phi_{s,r}(u,\xi)\;,\qquad \forall (u,\xi) \in \bU \times \CM\;.
\end{equ}
The event $\Phi_{s,t}(u,\mod)=\death$ should be
understood as the system ``blowing up''. After blow-up, the dynamics
can no longer be defined.  Consistent with this, we assume that
$\Phi_{s,t}(\death,\xi)=\death$ for all $\xi$  and all $0\leq s \le t \leq 1$, 
which embodies the belief that
``resurrection is impossible.'' We also use the shorthands
$\Phi_t = \Phi_{0,t}$ and $\Phi = \Phi_{0,1}$.

\begin{remark}
The property \eqref{e:consistency} may look strange to someone used to random
dynamical systems since we do not perform any time-shift of the element $\xi$.
Since we never look at times beyond time $1$, this turns out to be a much more
convenient convention in our setting.
\end{remark}

Throughout this article, we make the following rather strong regularity assumption
on the maps $\Phi_{s,t}$.

\begin{assumption}\label{ass:cont}
The preimage of $U$ under the map
\begin{equ}
\Phi\colon (s,t,u,\xi) \mapsto \Phi_{s,t}(u,\xi)\;,
\end{equ}
is open and $\Phi$ is jointly continuous on $\Phi^{-1}(U)$. 
Furthermore, $\Phi$ is Fr\'echet differentiable in $u$ at every point of 
$\Phi^{-1}(U)$.
\end{assumption}

In view of this assumption, we define the sets
\begin{equ}
\CN_t = \{(u,\xi)\,:\,  \Phi_t(u, \xi) \neq\death\}\;,\qquad \CN \eqdef \CN_1\;.
\end{equ}
In particular, our assumption guarantees that the sets $\CN_t$ are all open in 
$U \times \CM$. 
We will denote the Fr\'echet derivative of $\Phi_t$ in the
direction $v \in U$ by   $D\Phi_t(u,\xi)v$, with the understanding that 
$D\Phi_t$ is only defined on $\CN_t$.
We also assume that we are given a somewhat more quantitative way of 
measuring how close a given point $(u,\xi)$ is to being outside of $\CN_t$.
This is encapsulated in the following assumption.

\begin{assumption}\label{ass:r}
We are given a lower semi-continuous map
$r \colon [0,1]\times  U \times \CM \to [0,\infty]$ with the 
following properties.
\begin{enumerate}
\item For every $u \in U$ and every $t \in [0,1]$, the map 
$\omega \mapsto r_t(u,\mod(\omega))$ is $\CF_t$-measurable. 
\item For every $(u,\xi) \in U \times \CM$, one has
$r_0(u,\xi) < \infty$ and the map 
$t \mapsto r_t(u,\xi)$ is continuous and increasing. 
\item One has 
$\{(u,\xi)\,:\, \Phi_t(u, \xi) = \death\} =
\{(u,\xi)\,:\, r_t(u, \xi) = \infty\}$.
\item For every $t \in (0,1]$, the map $r_t$ is locally Lipschitz continuous
on $\CN_t$.
\end{enumerate}
\end{assumption}

The continuity of the map $t \mapsto r_t$ combined with the above
assumptions imply that
$\mathcal{N}_s \supset \mathcal{N}_t$ for all $0\leq s \leq t \leq 1$
and that the following continuity property holds
\begin{equ}
\mathcal{N}_t = \bigcup_{a \in (t,1]} \mathcal{N}_a \quad \text{for all $0\leq t \leq 1$}\;.
\end{equ}
In particular, this implies that for any $u \in U$
\begin{equ}
  \lim_{t \rightarrow 0^+} \P\big( (u,\mod) \in \mathcal{N}_t \big) =1\;.
\end{equ}
So far, we have described a rather generic class of random dynamical systems.
The only part of our setup that is somewhat non-standard is the 
factorisation of our map into a fixed measurable map $\mod$ which has a nice
Gaussian probability space as its domain and a map $\Phi$
which has nice continuity properties.
This factorisation is essential to the pathwise approach to stochastic analysis,
be it via rough paths \cite{Terry,Peter} or regularity structures \cite{reg}. Indeed,
for the examples we have in mind, it is in general not possible to simplify
the above setup by assuming that $\CM$ is a topological vector space and
the map $\mod$ is linear.

In the next subsection, we introduce the additional structural assumptions
that are crucial in guaranteeing the non-degeneracy of the noise
required for the strong Feller property to hold.


\subsection{The shift map and differentiation with respect to the noise}
\label{sec:shift-map}

We will be interested in transformations of the dynamics induced
by shifting the noise by translation. 
The assumptions we present in this section are twofold. On the one hand,
we want to exhibit a space of Cameron-Martin shifts which are sufficiently ``nice'' 
to be compatible with the topology of the space $\CM$. On the other hand, we want
to guarantee that this space is sufficiently large to contain shifts allowing
to compensate for arbitrary shifts in the initial condition. 


Regarding the first type of assumption, we consider a space 
$\V = L^p([0,1], X_0) \subset \CH$ of
directions, where $X_0$ is some separable Banach subspace of $H_0$ and $p \in (2,\infty)$
is some fixed value. (In the sequel, we will typically choose $p$ very large. The main reason
not to use $p = \infty$ is the lack of separability.) 
It will sometimes be convenient to also consider shifts in 
$\V_s  = L^p([0,s], X_0)$ for some $s \in (0,1]$. We then view
$\V_s$ as a subspace of $\V = \V_1$ by identifying its elements with
the corresponding function that vanishes on $(s,1]$.
For $\bar s \le s$, we also write $\pi_{\bar s}\colon \V_s \to \V_{\bar s}$ for the
restriction operator.
Our assumption regarding the compatibility of $\V$ with the topology of $\CM$ is
then as follows.

\begin{assumption}\label{ass:action}
We are given a \textit{continuous} action
$\tau\colon \V \times \CM \rightarrow \CM$
of $\V$ onto $\CM$ such that, for every $\xi \in \CM$,
the map $h \mapsto \tau(h,\xi)$ is locally Lipschitz continuous and such that, 
for every $h \in \V$, the identity
\begin{equ}[e:definition]
  \mod(\omega+ h)=   \tau(h,\mod(\omega))\;,
\end{equ}
holds $\P$-almost surely. 
Furthermore, the action $\tau$ is compatible 
with the maps $\Phi_{s,t}$ in the sense that if $h$ is such that 
$h(r) = 0 $ for $r \in [s,t]$, then
\begin{equ}[e:dependency]
\Phi_{s,t}(u,\tau(h,\xi)) = \Phi_{s,t}(u,\xi)\;,
\end{equ}
for every $u \in \bar U$.
\end{assumption}

Since we now have a notion of a shift on $\CM$, it is
reasonable to explore differentiation of $\Phi(u,\xi)$ with respect to
$\xi$. We do not want to assume full Malliavin differentiability,
which would essentially amount to defining a directional derivative for all
directions in
$\CH$, although in many cases of interest one does expect this to be 
achievable, albeit at the cost of significant additional technical difficulties, 
see for example \cite{PeterMalliavin}. 
Instead, we will only assume that the shift is well
behaved in the directions given by $\V \subset \CH$ as follows.

\begin{assumption}\label{ass:MallDer}
For every $(u,\xi) \in \CN_t$, the map
\begin{equ}
h \mapsto \Phi_t(u,\tau(h,\xi))\;,
\end{equ}
is Fréchet differentiable (as a map between the Banach spaces $\V$ and $U$)
at $h=0$. 
\end{assumption}

We denote this Fréchet derivative by
$\mathcal{D}\Phi_t(u,\xi)$.
Lastly, and this is the main assumption guaranteeing that the strong Feller property holds,
we would like to assume that the range of $\mathcal{D}\Phi_t(u,\xi)$ contains
that of $D \Phi_t(u,\xi)$. 
More precisely, we would like to find a linear operator $A_t(u,\xi)$ transferring a 
variation $v \in U$ in the initial condition $u$ to a
variation $h \in \V_t$ in the driving noise. 

However, since we furthermore want to use adapted shift maps,
the precise formulation of this condition requires some care. The formulation
we choose to go with is as follows.

\begin{assumption}\label{ass:shiftmap}
For every $s \le t$ with $t\in (0,1]$, we are given a map 
\begin{equ}
  A_t^{(s)} \colon \CN_s \rightarrow L(U, \V_s)\;,\footnote{Here,
  $L(U,\V)$ denotes the space of bounded linear operators $U \rightarrow \V$.}
\end{equ}
and these maps are compatible in the sense that, for
any $0 < s < r \le t$, any $(u,\xi) \in \CN_r$, and any $v \in U$, one has
\begin{equ}
A_t^{(r)}(u,\xi) v \restr [0,s]  = A_t^{(s)}(u,\xi) v\;.\footnote{As usual, we write $F\restr A$ for the restriction of a map $F$ to a subset $A$ of its domain.}
\end{equ}
Furthermore, for every $u \in U$ and $s \le t$, the map
\begin{equ}
\omega \mapsto A_t^{(s)}(u,\mod(\omega))\;,
\end{equ}
is $\CF_s$-measurable and  one
has the identity
\begin{equ}[e:propertyShift]
  D\Phi_t (u,\xi) v + \mathcal{D}\Phi_t(u,\xi) (A_t^{(t)}(u,\xi) v) = 0\;,
\end{equ}
for all $v \in U$ and all $(u,\xi) \in \CN_t$.
Furthermore, for every $0 < s \le t \le 1$, the map 
$A_t^{(s)}$ is locally Lipschitz continuous from $\CN_s$ to $L(U,\V_s)$
and bounded on every set of the form $\{(u,\xi)\,:\,r_s(u,\xi) \le R\}$
with $R > 0$.
\end{assumption}

\begin{remark}
This formulation seems somewhat awkward: why not simply define a map 
$A_t$ on $\CN_t$ satisfying \eqref{e:propertyShift} and be done with it?
The problem is that the event $(u,\mod(\omega)) \in \CN_t$ is in general not
$\CF_s$-measurable for $s < t$, so that our measurability condition would
fail in this case. On the other hand, Assumption~\ref{ass:r} does guarantee that
the event $(u,\mod(\omega)) \in \CN_t$ is $\CF_t$-measurable, so that 
the condition as stated has at least some chance of being fulfillable.
We will however sometimes use the notational convention $A_t = A_t^{(t)}$ 
for simplicity.
\end{remark}

\begin{remark}
We say that a map is ``locally Lipschitz'' if, for every $x$ in its domain,
there exists a neighbourhood of $x$ on which it is Lipschitz continuous.
\end{remark}

\section{Strong Feller property}
\label{sec:SF}

We now show how one can ``integrate up'' the infinitesimal shift given in \eqref{e:propertyShift} 
to produce a macroscopic shift which
drives  the solution starting from one initial condition to the same
location at the terminal time as the solution starting from a
nearby initial condition.  This pathwise coupling will then be used to establish the
strong Feller property. Since we will allow the existence of
trajectories which blow up, we will only be able to prove the
existence of this ``compensating shift'' on a set of positive (though not
necessarily full) measure of noise trajectories.

Throughout this section, we assume that we are given $\Phi$ and $\mod$ as above
and that Assumptions~\ref{ass:cont}--\ref{ass:shiftmap} are satisfied. We will henceforth 
simply refer to this as ``the context of Section~\ref{sec:setup}''. 
For any bounded measurable map 
$\Psi\colon \bar U \to \R$, we set $\PP \Psi(u) = \E \Psi(\Phi(u,\mod(\omega)))$,
which defines a Markov operator on $\bar U$.
The main abstract result of this article is that in the context of the previous section, 
the Markov operator $\PP$ enjoys the strong Feller property, namely $\PP\Psi$ is 
continuous for all $\Psi$ that are only bounded and measurable.
Our proof actually gives the apparently stronger conclusion that the transition 
probabilities of $\PP$ are continuous in the total variation distance, but this is known
\cite{DM83,Sei02,NotesNE} to be essentially equivalent.

\subsection{Existence of a compensating shift}
\label{sec:exist-comp-shift}

We now prove the existence of a ``compensating shift'' or rather a
shift in the noise so that the solution starting  from one initial
condition coincides with the solution starting from another initial
condition at time $t$. This is the content of equation
\eqref{e:defShift}  in the following theorem.

\begin{theorem}\label{theo:shift}
In the context of Section~\ref{sec:setup}, for every $u \in U$, every $M$ sufficiently large 
and every $t \in (0,1)$, there exists a map 
$(\bar u,\xi) \mapsto h^{(t,M)}(u,\bar u,\xi) \in \V_t$ such that
\begin{claim}
\item[1.] For every $(u,\xi) \in \CN_t$, there exist $M_+,\gamma_+ > 0$ such that the identity
  \begin{equ}[e:defShift]
    \Phi_t(u,\xi) = \Phi_t(\bar u, \tau(h^{(t,M)}(u,\bar u,\xi),\xi))\;,
  \end{equ}
  holds for every $M \ge M_+$ and for every $\bar u$ with $\|\bar u - u\|\le \gamma_+$.
\item[2.] For every $u$, $\bar u$ in $U$ and every $s \in (0,t)$, the 
map $\omega \mapsto \pi_s h^{(t,M)}(u,\bar u,\mod(\omega))$ is $\CF_s$-measurable.
\item[3.] One has the bound $\|h^{(t,M)}(u,\bar u,\xi)\|_\CH \le M\|\bar u - u\|_U$.
\end{claim}
\end{theorem}

\begin{proof}
We begin by defining a ``truncated'' version of the maps $A_t$. For this, we 
choose a smooth function $\chi\colon \R_+ \to [0,1]$ such that  $\chi(r) = 1$ for
$r \le 1$ and $\chi(r) = 0$ for $r \ge 2$, and we set
\begin{equs}
\bigl(\tilde A_t(u,h,\xi)v\bigr)(s) &=
\bigl(A_t^{(s)}(u,\tau(h,\xi))v\bigr)(s)\;,\\
\bigl(Z^R(u,h,\xi)\bigr)(s) &=
\chi(r_s(u,\tau(h,\xi))/R)\;,
\end{equs}
for all $s \in [0,t]$. As a consequence of Assumption~\ref{ass:shiftmap}, for every $s \in (0,t)$, 
the restriction of the map
$\tilde A_t(u,h,\xi)$ to $[0,s]$
is locally Lipschitz continuous as a function of $h$ and $\xi$ on the set
$\{(h,\xi)\,:\, (u,\tau(h,\xi)) \in \CN_s\}$.
For any direction $v \in U$ with $\|v\| = 1$ and for any $R > 0$, we then 
consider the ODE in $U\times \V$ given by
\begin{equs}[2][eq:pathODE]
\d_\gamma u_\gamma &= v\;,&\qquad u_0 &= u\;,\\
\d_\gamma h_\gamma &= \tilde A_t(u_\gamma,h_\gamma,\xi)v \inf_{\bar \gamma \le \gamma}
Z^R(u_{\bar \gamma},h_{\bar \gamma},\xi) \;,&\qquad h_0 &= 0\;.
\end{equs}
We claim that this ODE has unique global solutions. Indeed, 
consider $s$ such that $(u,\xi) \in \CN_s$. Then, it follows from Assumptions~\ref{ass:action}
and \ref{ass:shiftmap} 
that the ODE in $U\times \V_s$ given by \eqref{eq:pathODE}, but with 
the right hand side interpreted as taking values in $\V_s$, admits unique solutions
$(u_\gamma, h^{(s)}_\gamma)$ up to
the first value of $\gamma$ such that $r_s(u_\gamma,\tau(h^{(s)}_\gamma,\xi)) = 2R$.
Let us denote this value by $\gamma^\star_s$ (which also depends on $(u,\xi)$ of course).
Note that since $s \mapsto r_s$ is increasing, the map $s \mapsto \gamma^\star_s$
is decreasing. Note also that for $\bar s \le s$ and for $\gamma \le \gamma^\star_s$,
one has $\pi_{\bar s} h^{(s)}_\gamma = h^{(\bar s)}_\gamma$.
The desired solution to \eqref{eq:pathODE} is then given by
\begin{equ}
h_\gamma(s) = h_{\gamma \wedge \gamma^\star_s}^{(s)}(s)\;.
\end{equ}

We now show that the choice $h^{(t,M)}(u,\bar u,\xi) = h_\gamma$ indeed satisfies \eqref{e:defShift}
for $\bar u = u+\gamma v$ with $\|v\| = 1$, provided that $\gamma$ is small enough
and that $R$ is such that $\tilde A_t$ is bounded by $M$ on the set 
$\{(u,h,\xi)\,:\,r_t(u,\tau(h,\xi)) \leq R\}$.

Since $(u,\xi) \in \mathcal{N}_t$, we can find an $R$ large enough so
$r_t(u,\xi) \leq R/2$. Once this value of $R$ is fixed, we can find $\gamma_+ > 0$ such that 
$r_t(u,\tau(h,\xi)) \leq R$ for every $h$ with $\sup_{s\in [0,t]} \|h(s)\| \le \gamma_+$.
Since $\tilde A_t$ is bounded by $M$, we can furthermore choose $\gamma_+$ sufficiently small
so that $\tilde A_t \le 1/\gamma_+$.
This choice of $\gamma_+$ guarantees \textit{a priori} that, for $\gamma \le \gamma_+$,
the solution to \eqref{eq:pathODE} satisfies
\begin{equ}[e:solutionv]
\d_\gamma h_\gamma = A_t^{(t)}(u_\gamma,\tau(h_\gamma,\xi))v \;.
\end{equ}
Differentiating 
the quantity $\Phi_t(u_\gamma, \tau(h_\gamma,\xi))$ with respect to $\gamma$,
we then obtain from the definition of $\CD \Phi_t$, the chain rule, \eqref{e:solutionv},
and \eqref{e:propertyShift} that
\begin{equs}
{d\over d\gamma} \Phi_t(u_\gamma, \tau(h_\gamma,\xi))
&= D\Phi_t(u_\gamma, \tau(h_\gamma,\xi))v 
+ \CD\Phi_t(u_\gamma, \tau(h_\gamma,\xi))A_t^{(t)}(u_\gamma,\tau(h_\gamma,\xi))v \\
&= 0\;,
\end{equs}
which shows that the first property of $h$ does indeed hold.

The measurability of $\pi_s h^{(t)}$ with respect to $\CF_s$ is then 
an immediate consequence of the fact that both $A_t^{(s)}$ and $r_s$
are $\CF_s$-measurable, and the last property holds by definition 
since $\|\d_\gamma h_\gamma(s)\| \le M$ for $\gamma \le \gamma_s^\star$ and 
$\d_\gamma h_\gamma(s) = 0$ afterwards.
\end{proof}

\subsection{The Markov operator $\PP$ is strong Feller}
\label{sec:cp-strong-feller}

We now use the estimate given in \eqref{e:defShift} to prove that
$\PP$ satisfies the
strong Feller property. The existence of a compensating shift as in
\eqref{e:defShift} can be combined with Girsanov's theorem to show that
nearby points have transition measures which are close in total
variation. 

\begin{theorem}\label{theo:main}
In the context of Section~\ref{sec:setup}, the Markov operator $\PP$ 
satisfies the strong Feller property.
\end{theorem}

\begin{proof}
Without loss of generality,
we assume that $\Psi(\death) = 0$ and that $\sup_{u \in \bar U} |\Psi(u)| \le 1$.
It is sufficient to show that, for any fixed $u \in U$, one has
\begin{equ}[eq:SFenough]
\lim_{\gamma \to 0}\sup_{\|\bar u - u\|\le \gamma} |\PP \Psi(u) - \PP \Psi(\bar u)| = 0\;. 
\end{equ}
For this, we proceed as follows. Viewing the initial condition $u$ and
any time $t \in (0,1)$ as fixed, we choose a direction $v \in U$
and write $h_{\gamma,t}^{(M)}(\xi,v)$ for the corresponding shift
obtained in Theorem~\ref{theo:shift}, namely
\begin{equ}
h_{\gamma,t}^{(M)}(\xi,v) = h^{(t,M)}(u,u+\gamma v,\xi)\;.
\end{equ}
We also fix some $\eps > 0$.

It follows from Theorem~\ref{theo:shift} that
we can find $M_+,\gamma_+ > 0$ (possibly depending on $t$ and $\xi$) such that 
the identity
\begin{equ}
\Phi_t(u,\xi) = \Phi_t(u + \gamma v, \tau(h_{\gamma,t}^{(M)}(\xi,v),\xi))\;,
\end{equ}
and the bound $\|h_{\gamma,t}^{(M)}(\xi,v)\|_{\CH} \leq \gamma M$
hold for all $\gamma \leq \gamma_+(\xi,t)$, all $M \ge M_+(\xi,t)$, and all $v$ with $\|v\| = 1$.
In particular, as a consequence of \eqref{e:consistency} and \eqref{e:dependency}, one has
\begin{equ}
\Phi\big(u,\xi\big) = \Phi\big(u + \gamma v, \tau(h_{\gamma,t}^{(M)}(\xi,v),\xi)\big)\;,
\end{equ}
since if the trajectories coincide at time  $t$ then they   coincide
for all $T >t$.

Fixing a $u \in U$, $v \in U$ with $\|v\| = 1$, we write $\bar u
= u+\gamma v$  for $\gamma > 0$. We also 
define the sets
\begin{equ}
\obar \Omega_{\gamma,M,t} = \{\omega\,:\, \gamma \le
\gamma_+(\mod(\omega),t)\quad M \ge M_+(\mod(\omega),t)\}\;,
\end{equ}
as well as
\begin{equ}
 \Omega_{\gamma,M,t} = \{\omega\,:\, \omega + h_{\gamma,t}^{(M)}(\mod(\omega),v)
 \in \obar\Omega_{\gamma,M,t} \}\;.
\end{equ}
(The latter actually also depends on $v$, but this will be irrelevant for our proof since
all of our estimates are uniform in $v$.)
Observe that since both $\gamma_+$ and $M_+$ are finite for all $(u,\xi) \in \CN_t$,
one has that
$\P(\obar \Omega_{\gamma,M,t} )\rightarrow \P((u,\mod) \in \CN_t)$ as $\gamma
\rightarrow 0$ and $M \to \infty$. Furthermore, $\P((u,\mod) \in \CN_t) \rightarrow 1$ as
$t \rightarrow 0$. Combining these facts implies for any $\eps>0$
there exists a $t_0$ and, for every $t \le t_0$, there exist $\gamma_0$ and $M_0$ so that
\begin{equ}[eq:NotOmegaIsSmall]
 \P( \obar\Omega_{\gamma,M,t}^c) \leq \eps\quad  \text{for any $t \in (0,t_0]$, 
 $\gamma \le \gamma_0$ and $M \ge M_0$,}
\end{equ}
where we wrote  $B^c$ for the complement
of an event $B \subset \Omega$.
We will henceforth always assume that $\gamma$ is sufficiently small so that $\gamma M \le 1$.

Next, observe that we can then write
\begin{equs}
\PP\Psi(u) &- \PP\Psi(\bar u) = \E \bigl(\Psi(\Phi(u,\mod)) - \Psi(\Phi(\bar u,\mod))\bigr)\label{e:diff} \\
&\quad = \E \Bigl(\Psi(\Phi(u,\mod)) \bigl(\one_{\Omega_{\gamma,M,t}} + \one_{\Omega_{\gamma,M,t}^c}\bigr)
- \Psi(\Phi(\bar u,\mod))\bigl(\one_{\obar \Omega_{\gamma,M,t}} + \one_{\obar \Omega_{\gamma,M,t}^c}\bigr) \Bigr)\;.
\end{equs}
Using first Girsanov's theorem, then the definition of
$\Omega_{\gamma,M,t}$ combined with Proposition~\ref{prop:shift} below and
finally \eqref{e:defShift}, we obtain the identity
\begin{equs}
\E \Psi(\Phi(\bar u,\mod))&\one_{\obar\Omega_{\gamma,M,t}(\mod)}
= \E \Psi(\Phi(\bar u,\mod(\omega + h^{(M)}_{\gamma,t}(\mod)))) \CE_\gamma(\omega,v) 
\one_{\obar \Omega_{\gamma,M,t}}(\omega + h^{(M)}_{\gamma,t}(\mod)) \\
&= \E \Psi(\Phi(\bar u, \tau(h^{(M)}_{\gamma,t}(\mod),\mod))) \CE_\gamma(\omega,v) 
\one_{\Omega_{\gamma,M,t}}(\omega) \\
&= \E \Psi(\Phi(u, \mod)) \CE_\gamma(\omega,v) 
\one_{\Omega_{\gamma,M,t}}(\omega) \\
&= \E \Psi(\Phi(u, \mod))  
\one_{\Omega_{\gamma,M,t}}(\omega) 
+ \E \Psi(\Phi(u, \mod)) \bigl(\CE_\gamma(\omega,v)-1\bigr)
\one_{\Omega_{\gamma,M,t}}(\omega)\;.
\end{equs}
where  $\CE_\gamma(\omega,v)$ is the exponential martingale given
by Girsanov's theorem.  Rearranging this equality  and using the bound $|\Psi|
\leq 1$ produces
\begin{equ}
 \Big| \E \Psi(\Phi(\bar u,\mod))\one_{\obar \Omega_{\gamma,M,t}}(\mod) - \E \Psi(\Phi(u, \mod))  
\one_{\Omega_{\gamma,M,t}}(\omega) \Big|\le \sqrt{\E \bigl(\CE_\gamma(\omega,v)-1\bigr)^2}
\end{equ}
Since $\|h^{(M)}_{\gamma,t}(\mod,v)\|_\CH \le 1$ by Theorem~\ref{theo:shift}
and the assumption $\gamma M \le 1$, one has
\begin{equ}
\E \bigl(\CE_\gamma(\omega,v)-1\bigr)^2 = \E \CE_\gamma(\omega,v)^2 -1
= \E \exp(\|h^{(M)}_{\gamma,t}(\mod,v)\|_\CH^2) - 1
\le e \E \|h^{(M)}_{\gamma,t}(\mod,v)\|_\CH^2\;.
\end{equ}
Similarly, using again Girsanov's theorem,  if $\mathcal{A}=\{ \omega :
(u,\mod) \in \mathcal{N}\}$ and $\mathcal{\bar A}=\{ \omega :
(\bar u,\mod) \in \mathcal{N}\}$ we also obtain the bound
\begin{equs}
\P (\Omega_{\gamma,M,t}^c) =
\E \bigl(\CE_\gamma \one_{\Omega_{\gamma,M,t}^c}\bigr) -\E \bigl((\CE_\gamma-1) \one_{\Omega_{\gamma,M,t}^c}\bigr)
 &=
\P \bigl(\obar\Omega_{\gamma,M,t}^c\bigr) -\E \bigl((\CE_\gamma-1) \one_{\Omega_{\gamma,M,t}^c}\bigr) \\&\le 
\P (\obar \Omega_{\gamma,M,t}^c)+\bigl(\E (\CE_\gamma-1)^2\bigr)^{1/2}\;.
\end{equs}
Combining all of these bounds with \eqref{e:diff}  and $|\Psi| \leq 1$, we finally obtain that 
\begin{equs}
|\PP\Psi(u) - \PP\Psi(\bar u)| &\le \big|\E \Psi(\Phi(u,\xi)) \one_{\Omega_{\gamma,M,t}}- \E\Psi(\Phi(\bar u,\xi))\one_{\obar \Omega_{\gamma,M,t}}  \big|\\
&\qquad + \P(\Omega_{\gamma,M,t}^c)
+\P(\obar \Omega_{\gamma,M,t}^c)\;,\\
&\le
2 \P (\obar\Omega_{\gamma,M,t}^c) + 2e \sqrt{\E \|h^{(M)}_{\gamma,t}(\mod,v)\|_\CH^2}\;.\label{eq:lastBound}
\end{equs}
Fixing any $\eps>0$ from \eqref{eq:NotOmegaIsSmall} we can find a
$t$ small enough so that the first term above is less than
$\eps/2$ for $M$ sufficiently large and 
$\gamma$ sufficiently small. Then, possibly choosing $\gamma$ even smaller
so that $2e\gamma \le \eps/2$ and $\gamma M \le 1$, the second term
will also be less that $\eps/2$. Hence we conclude that for all
$\gamma$ sufficiently small one has the bound
\begin{equ}
 \sup_{\|\bar u - u\|\le \gamma}  |\PP\Psi(u) - \PP\Psi(\bar u)| \le \eps\;.
\end{equ}
As $\eps$ was arbitrarily, we conclude that
\eqref{eq:SFenough} holds and the proof is complete.
\end{proof}
\begin{remark}
  If one knows that the system is not explosive then one can often obtain a
  quantitative bound in \eqref{eq:lastBound} and show that $\PP \Psi$
  is H\"older continuous  whenever  $\Psi$ is bounded. In particular, if $\P( \mod
  \in \mathcal{N}_t)=1$ for some $t \in (0,1]$ then one is free to fix
  that $t$ in  \eqref{eq:lastBound}. Having fixed $t$, one knows from
  Theorem~\ref{theo:shift} that 
  \begin{equ}
    \sqrt{\E \|h^{(M)}_{\gamma,t}(\mod,v)\|_\CH^2} \leq C \gamma = C
    \|u-\bar u\|\;.
  \end{equ}
Combining this with any estimate on $|\P (\obar\Omega_{\gamma,M,t})
- 1|$ as $\gamma \rightarrow 0$, such as $\P
(\obar\Omega_{\gamma,M,t}) \leq K\gamma^p$ for some $p >0$ will produce
a quantitative estimate. However, as such a quantitative bound is not needed to prove unique
ergodicity of the system,  we do not dwell on this point.
\end{remark}

\begin{remark}
  While we have emphasised the use of \eqref{e:defShift} at points which
  do not blow up, our result also shows that the probability
  of blowing up  varies continuously with the initial condition:
  just apply $\PP$ to the test function $\Psi$ given by $\Psi(\death) = 1$
  and $\Psi(u) = 0$ for $u \in U$. The intuition is clear: if there is
  a compensating shift which shifts the trajectory from a nearby
  initial condition to coincide with a trajectory which is en route to
  exploding, then the nearby point also blows up with positive probability.   
\end{remark}

\subsection{Robustness of shift maps}

In the proof given in the previous subsection, we used the fact that \eqref{e:definition} 
also holds for \textit{random} shifts $h$
as long as they are adapted. The aim of this section is to show that 
this is indeed automatically the case, at least under some Novikov-type
condition. (This could probably be relaxed, but it is trivially satisfied in
our application since our cut-off procedure guarantees that we only consider 
shifts that are bounded by a deterministic
constant.) We will make use of the following fact.

\begin{lemma}\label{lem:RV}
Let $\CX$ and $\CY$ be two
Polish spaces, $\P$ a probability measure on $\CX$, and $X$, $Y$
two measurable maps $\CX \to \CY$. Then $X = Y$ $\P$-almost surely if and only if
the identity
\begin{equ}
\E f(X(x))g(x)
= \E f(Y(x))g(x)\;,
\end{equ}
holds for any two bounded \textit{uniformly continuous} functions $f\colon \CY \to \R$ and 
$g\colon \CX \to \R$.
\end{lemma}

The main result of this subsection is the following. Here, for any separable Banach 
space $\CX$, $L^p_\ad(\Omega \times [0,1], \CX)$ denotes the set of $L^p$ functions
from $\Omega \times [0,1]$ into $\CX$ that are measurable with respect to the predictable
$\sigma$-algebra on $\Omega \times [0,1]$ (with respect to the filtration $\{\CF_t\}_t$).
In the particular case $\CX = X_0$, one has of course 
$L^p_\ad(\Omega \times [0,1], \CX) \subset L^p(\Omega \times [0,1], \CX)
\approx L^p(\Omega, \V)$, so that we alternatively view these elements as $\V$-valued random variables.

\begin{proposition}\label{prop:shift}
Let $\mod\colon\Omega \to \CM$ and let $\tau$ be a continuous action of $\V$ on $\CM$
as above. If \eqref{e:definition} holds for every deterministic $h \in \V$, then it also holds
for every predictable $h\colon \Omega \to \V$ satisfying $\E \exp(6\|h\|_\CH^2)< \infty$.
\end{proposition}

\begin{proof}
Let $h_n \in L^p_\ad(\Omega \times [0,1], X_0)$ be a sequence of random variables such that
each $h_n$ takes only countably many values in $\V$, and such that the $h_n$ converge to $h$ 
both in $L^2_\ad(\Omega \times [0,1], H_0)$ and in probability in $\V$.
We choose the $h_n$ furthermore such that $\sup_n \E \exp(6\|h_n\|_\CH^2)< \infty$.
This is always possible by the separability of $X_0$ and the density of 
simple processes in $L^p_\ad(\Omega \times [0,1], X_0)$.

Since $h_n$ only takes countably many values, 
the identity
\begin{equ}
  \mod(\omega+ h_n(\omega))=   \tau(h_n(\omega),\mod(\omega))\;,
\end{equ}
holds almost surely for every $n$. Here, assuming that $h_n$ takes the value $h_n^{(k)}$ on
the set $A_k \subset \Omega$ forming a partition of $\Omega$, 
the left hand side should be interpreted as the random variable
which equals $\mod(\omega+ h_n^{(k)})$ on $A_k$ for every $k$.
In particular, for every $n$, we have the identity
\begin{equ}[e:identity]
  \E f(\mod(\omega+ h_n(\omega))) g(\omega+ h_n(\omega)) \CE(h_n)=   \E f(\tau(h_n(\omega),\mod(\omega)))g(\omega+ h_n(\omega))\CE(h_n)\;,
\end{equ}
where we used the notation
\begin{equ}
\CE(h) \eqdef \exp\Big(-\int_0^1 \scal{h(s),d\omega(s)}_\CH- {1\over 2} \int_0^1 \|h(s)\|_\CH^2\,ds\Big)\;.
\end{equ}
It follows from the first and second properties that $\CE(h_n)$ converges 
towards $\CE(h)$ in $L^1(\Omega,\P)$: writing $\CE(h) = \exp(\CI(h))$, it is easy
to verify that
\begin{equs}
\E |\CI(h)&-\CI(h_n)|^2 \le \E \|h-h_n\|_\CH^2 \bigl(1+\|h+h_n\|_\CH^2\bigr) \\
&\lesssim \bigl(\E \|h-h_n\|_\CH^2 \E \bigl(1+\|h\|_\CH + \|h_n\|_\CH\bigr)^6\bigr)^{1/2} 
\lesssim \bigl(\E \|h-h_n\|_\CH^2\bigr)^{1/2} \;.
\end{equs}
We conclude that 
\begin{equs}
\E |\CE(h) &- \CE(h_n)| \le \E |\CI(h)-\CI(h_n)| \bigl(\CE(h) + \CE(h_n)\bigr) \\
&\le \Bigl(2\E |\CI(h)-\CI(h_n)|^2 \E \bigl(\CE(h)^2 + \CE(h_n)^2\bigr)\Bigr)^{1/2}
\lesssim \bigl(\E |\CI(h)-\CI(h_n)|^2\bigr)^{1/4}\;,
\end{equs}
which does indeed converge to $0$ by assumption.

As a consequence of the uniform continuity of $g$ and $\tau$ and the 
convergence of
$h_n$ to $h$, the right hand side of \eqref{e:identity} converges as $n\to \infty$ to
\begin{equ}[e:rhs]
\E f(\tau(h(\omega),\mod(\omega)))g(\omega + h(\omega))\CE(h)\;.
\end{equ}
On the other hand, for any fixed $n$, we can use Girsanov's theorem
to show that the left hand side equals
\begin{equ}
 \E f(\mod(\omega)) g(\omega)\;,
\end{equ}
independently of $n$. Applying Girsanov's theorem in the opposite direction and comparing this to 
\eqref{e:rhs}, we conclude that
\begin{equ}
 \E f(\mod(\omega+ h(\omega))) g(\omega+ h(\omega)) \CE(h)
 = 
 \E f(\tau(h(\omega),\mod(\omega)))g(\omega + h(\omega))\CE(h)\;.
\end{equ}
Writing $\tilde \E$ for expectations under the probability measure $\tilde \P = \CE(h)\,\P$
and writing $\sigma_h$ for the $\sigma$-algebra generated by the shift map
$T_h$, it follows from Lemma~\ref{lem:RV} that the identity
\begin{equ}
f(\mod(\omega+ h(\omega)))
= \tilde \E \bigl(f(\tau(h(\omega),\mod(\omega)))\,|\, \tilde \sigma\bigr)\;,
\end{equ}
holds $\tilde \P$-almost surely.
Since this holds for every bounded uniformly continuous function $f$, it also holds for its square,
so that 
\begin{equ}
\tilde \E \bigl(f(\tau(h(\omega),\mod(\omega)))^2\,|\, \tilde \sigma\bigr)
= 
\bigl(\tilde \E \bigl(f(\tau(h(\omega),\mod(\omega)))\,|\, \tilde \sigma\bigr)\bigr)^2\;.
\end{equ}
This implies that $\omega \mapsto f(\tau(h(\omega),\mod(\omega)))$ is
$\tilde\sigma$-measurable, thus concluding the proof.
\end{proof}

\subsection{A few consequences}

Recall that every invariant measure for a Markov operator
$\PP$ on a Polish space can be written as a superposition of ergodic invariant measures.
In our case, the Dirac mass located on $\death$ is always invariant because of the property
$\Phi(\death,\xi) = \death$, but one might be interested in asking whether $\PP$ also
has a unique invariant measure located on $U$. The main ingredient for such a statement
is the fact \cite{DPZ} that the strong Feller property yields the following strengthening of the
fact that distinct ergodic invariant measures are mutually singular.

\begin{proposition}\label{prop:SF}
Let $\PP$ be a strong Feller Markov operator on a Polish space and let 
$\mu \neq \nu$ be ergodic invariant probability measures for $\PP$. Then,
the topological supports of $\mu$ and $\nu$ are disjoint. 
\end{proposition}

We immediately obtain the following two corollaries.

\begin{corollary}
In the context of Section~\ref{sec:setup}, assume that, for
every open set $A \subset U$ and every $u \in U$, one has
\begin{equ}[e:supp]
\P \bigl(\{\omega\,:\, \Phi(u,\mod(\omega)) \in A\}\bigr) > 0\;.
\end{equ}
Then, $\PP$ admits at most one invariant measure $\mu$ with $\mu(U) = 1$.
\end{corollary}

\begin{proof}
The condition \eqref{e:supp} implies that every invariant measure located on $U$
has full support, so that the claim follows from Proposition~\ref{prop:SF}.
\end{proof}

The next corollary is useful in situations where an invariant measure is
already known, as is the case for the KPZ equation \cite{FQ}, 
the 2D Navier-Stokes equations driven by space-time white noise \cite{DPD},
the dynamical $P(\Phi)_2$ model \cite{AlbRock91,DPD2}, as well as the dynamical
$\Phi^4_3$ model \cite{Konstantin}.

\begin{corollary}\label{cor:unique}
  In the context of Section~\ref{sec:setup}, assume that $\PP$ admits an invariant measure
$\mu$ with $\supp \mu = U$. Then $\mu$ is the only invariant measure 
concentrated on $U$.
\end{corollary}

\begin{proof}
Since this result does not appear to be easy to find in the literature, we give
a self-contained proof.
We first show that $\PP$ can have at most countably many ergodic invariant
measures as a consequence of the strong Feller property. Recall that the 
strong Feller property implies the existence of a probability measure $\lambda$ such
that the transition probabilities $\PP(u, \cdot)$ are absolutely continuous
with respect to $\lambda$ for every $u \in U$. (See for example \cite[Lem.~1.6.4]{NotesNE}.) 
In particular, by Prokhorov's theorem, 
there exists a $\sigma$-compact (i.e.\ countable union of compact)
set $K \subset U$ such that $\PP(u, K) = 1$ for every $u \in U$.

The strong Feller property also implies that for every $u \in U$ there exists a
neighbourhood $N_u$ of $u$ such that there exists at most one ergodic invariant measure
$\mu_u$ for with $\supp \mu \cap N_u \neq \emptyset$. By the $\sigma$-compactness of $K$,
we can find a sequence $\{u_n\}_{n \in \N}$ so that $K \subset \bigcup_{n \in \N} N_{u_n}$.
Since every invariant measure is concentrated on $K$, the claim follows.

Assume now by contradiction that $\mu$ is not ergodic. Then, by the above, we can write
$\mu = \sum_{n} p_n \mu_n$ for at most countably many distinct ergodic invariant
measures $\mu_n$. Furthermore, by Proposition~\ref{prop:SF}, the supports of the $\mu_n$
are all disjoint. Since the supports are furthermore closed, their union has to be
all of $U$ by our assumption on $\mu$, and $U$ is connected, we immediately conclude that
the sum can have only one term.
\end{proof}

\section{Application to singular SPDEs}
\label{sec:reg}

In this section, we show how to apply the previous abstract result to
solutions to singular stochastic PDEs of the type studied in \cite{reg}.
We henceforth assume that we are in the setting of \cite[Sec.~7.3]{reg}
which we summarise in the next subsection.

\subsection{A very general setting}

We are working with an ambient regularity structure
$(\CT,\CG)$ which contains the usual polynomial regularity structure 
$(\bar \CT, \bar \CG)$ on $\R^{d+1}$ with scaling $\s = (2q,\bar \s)$ for some scaling $\bar \s$ of $\R^d$.
In particular, we have a distinguished basis vector $\one \in \bar \CT$ of degree $0$ which 
represents the constant functions.
We assume that it is also endowed with an abstract integration map $\CI$ of order $2q$
as well as a convolution operator $\CP$ 
(associated to $\CI$ as in \cite[Sec.~4]{reg}), obtained from the linear 
evolution problem associated to a homogeneous (with respect to the scaling $\bar \s$) 
elliptic differential operator $\CL$
of order $2q$ with constant coefficients on $\R^d$. In particular, one has
\begin{equ}
\CR \CP \PPhi = P \star \CR \PPhi\;,
\end{equ}
for any $\PPhi \in \CD^\gamma$ (taking values in the domain of $\CI$) with $\gamma > 0$
and for $P$ the Green's function of $\d_t - \CL$. 
A very important property of the kernel $P$ we will use in the sequel is that it is 
non-anticipative in the sense that $P(t,x) = 0$ for $t < 0$.
Henceforth, we will also assume that 
we are given a lattice $L$ on $\R^d$ with compact fundamental domain 
and that all quantities of interest are periodic with respect 
to $L$. For example, we simply write $\CC^\alpha$ for the space of H\"older continuous
functions of regularity $\alpha$ that are $L$-periodic, without specifically mentioning $L$.
The same is true for our spaces $\CD^{\gamma,\eta}$ of modelled distributions.

As usual, we write $\MM$ for the space of all admissible models for $(\CT,\CG)$ 
and the scaling $\s$ that are furthermore periodic (in the sense of \cite[Def.~3.33]{reg}) with respect
to the lattice $L$. Here, a model $(\Pi,\Gamma)$ is said to be admissible if
\begin{claim}
\item it agrees with the polynomial model on $\bar \CT$; 
\item it realises $K$ for $\CI$ (in the sense of \cite[Def.~5.9]{reg}) for a kernel 
$K$ such that $P = K+R$ where $R$ is smooth and $K$ satisfies \cite[Ass.~5.1\,\&\,5.4]{reg}.
\end{claim}
Such a decomposition for $P$ exists by \cite[Lem.~5.5]{reg}.

\begin{remark}
Note that the operator $\CP$ here is equal to the operator $\CK_{\bar \gamma} + R_\gamma \CR$
in \cite[Thm~7.8]{reg}.
\end{remark}

\begin{remark}
In general, one can allow for several operators $\CI^{(k)}$ with $k$ taking values in 
some finite index set, associated to integral kernels $P^{(k)}$. 
All of the results in this section carry over to this more general setting without
any substantial change, but we prefer to stick to the scalar-valued case to simplify notations.
\end{remark}

We also assume that we are given two finite-dimensional sectors $V$ and $\bar V$ such that
\begin{equ}[e:propSectors]
\bar \CT \subset V \subset \bar \CT \oplus \CT_{\ge \zeta}\;,\qquad
\bar \CT \subset \bar V \subset \bar \CT \oplus \CT_{\ge \bar \zeta}\;,
\end{equ}
for some $0 < \zeta \le \bar \zeta + 2q$.
We also fix from now on exponents $\gamma, \bar \gamma$ such 
that $\bar \gamma + 2q > \gamma \ge \bar \gamma > 0$, such that $\bar \gamma \le \zeta$, and such that 
$\CQ_{< \gamma} \CI \bar V_{<\bar \gamma} \subset V_{<\gamma}$, where $\CQ_{< \gamma}$ denotes projection
onto $\CT_{<\gamma}$.

To describe the nonlinearity of our evolution problem, we 
assume that we are given a distinguished element $N \in \CT$ such that
$\Gamma N = N$ for every $\Gamma$ belonging to the
structure group $\CG$.
In particular, as a consequence of the fact that $N$ is invariant under $\CG$ and
$V$ contains $\bar \CT$, the subspace $\hat V \subset \CT$ given by
\begin{equ}
\hat V = V + \R\CI(N)\;,
\end{equ}
is again a sector. 
In some situations, the choice $N = 0$ is possible, but we will see in 
Section~\ref{sec:Phi43} below that one sometimes needs to make a different choice
in order to remain compatible with the condition $\zeta>0$ with $\zeta$ appearing
in the properties of $V$.
Given this sector $\hat V$, we make the following structural assumption
on both our models and our nonlinearity.

\begin{assumption}\label{ass:canonical}
We are given a continuous ``canonical lift map'' $\LL \colon \CC^\infty(\R\times \R^d) \to \MM$
as well as a finite-dimensional Lie group $\RR$, together with a continuous
action $M$ of $\RR$ onto $\MM$.

We are also given a map $F \colon \hat V_{<\gamma} \to \bar V_{<\bar \gamma}$ such that the 
corresponding composition operator 
is strongly locally Lipschitz continuous (in the sense of \cite[Sec.~7.3]{reg})
from $\CD^{\gamma,\eta}$ to $\CD^{\bar \gamma, \bar \eta}$, locally uniformly 
over the underlying admissible model, and for some exponents $\eta, \bar \eta > -2q$
such that $(\bar \eta \wedge \bar \zeta) + 2q > (\eta \vee 0)$
and such that $\eta \le \zeta$.
\end{assumption}

The first part of this assumption leads us to the following definition of a ``nice'' model.

\begin{definition}
An admissible model $\PPi = (\Pi,\Gamma) \in \MM$ for $(\CT,\CG)$ is \textit{nice} if
there exist $\noise_n \in \CC^\infty(\R\times \R^d)$ and $g_n \in \RR$ such that 
$\PPi = \lim_{n \to \infty} M_{g_n}\LL(\noise_n)$ 
and furthermore the distribution $P \star \Pi_z N$ 
belongs to $\CC(\R, \CC^{\eta})$.
\end{definition}

Write $\CM$ for the closure of all smooth and nice models in the 
space of nice admissible models for the regularity structure $(\CT,\CG)$. Setting $I_0 = [-2,3]$, 
we endow 
$\CM$ with the ``seminorm'' given by 
\begin{equs}
\$(\Pi,\Gamma)\$ &= \sup_{z,\phi,\lambda,\alpha,\tau} \lambda^{-\alpha} |(\Pi_z\tau)(\phi_z^\lambda)|
+ \sup_{z,\bar z,\alpha,\beta,\tau} |z-\bar z|_\s^{-(\alpha-\beta)} \|\Gamma_{z\bar z}\tau\|_{\CT_\beta}\\
&\qquad + \|P\star\Pi_0 N\|_{\CC(I_0, \CC^{\eta})}\;,
\end{equs}
as well as the corresponding distance function as in \cite[Eq.~2.17]{reg}.
Here,
the supremum over $z$, as well as the one over $z,\bar z$ run over 
$I_0 \times \R^d$, while the suprema over all other variables are as in \cite[Def.~2.17]{reg}.
We also identify elements that are at distance $0$, thus turning $\CM$ into a separable
metric space. Although, strictly speaking, an element of $\CM$ is now an equivalence class of
models, all of the operations we will ever make use of only ever concern 
modelled distributions defined on $[0,1] \times \R^d$ (which are canonically 
extended to be $0$ outside), so that both $\CP$ and the reconstruction operator are
well-defined and do not depend on the representative in $\CM$.

Writing $(\Omega,\P)$ for a filtered Gaussian probability space as in Section~\ref{sec:setup},
we furthermore assume without loss of generality that $\Omega$ is of the form 
$\CC^{-1/3}(\R,H)$ for some separable Hilbert space $H$ containing $H_0$ as a dense subspace
and such that both $H$ and $H_0$ are canonically identified with some space of
($L$-periodic) functions / distributions
on $\R^d$ in the sense that one has continuous and dense embeddings
\begin{equ}
\CC^\infty(\R^d) \subset H_0 \subset H \subset \CD'(\R^d)\;.
\end{equ}
In particular, elements of $\Omega$ can be viewed as distributions on $\R^{d+1}$.
We then make the following assumption.

\begin{assumption}\label{ass:model}
For every compactly supported mollifier $\rho \in \CC_0^\infty(\R^{d+1})$
there exists a sequence $g_\eps \in \RR$ such that the sequence of random models
$M_{g_\eps} \LL(\rho^\eps \star \omega)$ converges in probability in $\CM$ to a limiting
random model $\mod(\omega)$.

Furthermore, the model $\mod = (\Pi,\Gamma)$ is such that, for every $\bar t \le t \in [0,1]$,
the random variables $\bigl(\Pi_z \tau\bigr)(\phi_z)$ and $\Gamma_{z\bar z}$ are
$\CF_t$-measurable for every $\phi \in \CC_0^\infty$ with support in $\R_-\times \R^d$.
Here, the time coordinates of $z, \bar z$ are given by $t,\bar t$.
\end{assumption}

\begin{remark}
In all the examples we have in mind, the sequence $g_\eps$ can be chosen in
such a way that the limiting random model $\mod$ does not depend on the choice
of mollifier $\rho$. This property however is not essential for our analysis.
\end{remark}

It then follows from a combination of \cite[Thm~7.8]{reg} and the argument
given in \cite[Sec.~9]{reg} 
that, for every periodic initial condition $\Phi_0 \in \CC^\eta$
and every nice periodic admissible model, one has a maximal 
solution $\PPhi$ to the equation
\begin{equ}[e:SPDE]
\PPhi = \CP \one_+ \bigl(F(\PPhi) + N\bigr) + P \Phi_0\;,
\end{equ}
where $\one_+$ is the multiplication operator by the indicator function of the
set $\{(t,x)\,:\, t > 0\}$ and $P\Phi_0$ denotes the solution to the linearised
problem, viewed via its truncated Taylor expansion as an element in 
$\CD^{\gamma,\eta}$ with values in the usual Taylor polynomials $\bar \CT$.

In other words, there exists $T \in (0,2]$ (depending on the initial condition $\Phi_0$ and on the
underlying model $\PPi$) such that, for every $t < T$, there exists a unique
element $\PPhi \in \CD^{\gamma,\eta}((-\infty,t] \times \R^d, \hat V_{<\gamma})$ vanishing on 
$\R_- \times \R^d$ and such that 
the identity \eqref{e:SPDE} holds on $(-\infty,t] \times \R^d$. 
As a consequence of the model being nice and the condition $\zeta > 0$, 
the solution is furthermore such that
$\CR\PPhi$ takes values in $\CC((0,T), \CC^\eta)$ and, if
$T \le 1$, then $\lim_{t \to T} \|\big(\CR\PPhi\big)(t,\cdot)\|_{\CC^\eta} = \infty$.

\begin{remark}
The upper bound $2$ we impose on $T$ is of course completely arbitrary, as long as it is 
larger than the final time $1$ we are really interested in. However, since
our models are undefined outside the time interval $I_0$, we have to put some 
threshold otherwise the equation makes no longer sense.
\end{remark}

In particular, provided that we set $U = \CC^\eta$ (or rather the closure of smooth functions
in the $\CC^\eta$-norm so that $U$ is separable) and we similarly consider solutions
starting with an initial condition at time $s$ instead of time $0$, this does indeed yield
a collection of flow maps $\Phi_{s,t}\colon \bar U \times \MM \to \bar U$ 
which is consistent in the sense
of \eqref{e:consistency} as a consequence of \cite[Prop.~7.11]{reg}.
Here, we set
$\Phi_{s,t}(\Phi_s,\PPi) = \death$ if and only if the maximal existence time $T$
for the solution to \eqref{e:SPDE} with initial condition $\Phi_s$ at time $s$ and
model $\PPi \in \MM$ satisfies $T \le t$.

\subsection{A nice class of nonlinearities}

We first show that, under suitable assumptions on $F$, the solution
to \eqref{e:SPDE} also satisfies the remaining assumptions of Section~\ref{sec:setup}.
Our main assumption is that solutions to \eqref{e:SPDE} driven by nice models
are limits of solutions to stochastic PDEs with local nonlinearities.
Given a smooth function $u$ on $\R^{d+1}$ we write $G(\d^* u)$
for any expression of the form $G(u,\d_i u,\d^2_{ij} u, \ldots)$
which depends on \textit{finitely many} spatial derivatives of $u$. In our context,
this will only ever include derivatives of order strictly less than the order $2q$ of
the operator $\CL$. We then make the following assumption on $F$:

\begin{assumption}\label{ass:repr}
For every $g \in \RR$ there exist locally Lipschitz continuous functions 
$F_g^{(1)}$ and $F_g^{(2)}$ such that, for every model of the form
$\PPi = M_g \LL(\noise)$ with $\noise \in \CC^\infty(\R^{d+1})$, the solution $\PPhi$
to \eqref{e:SPDE} is such that $u = \CR \PPhi$ solves the PDE
\begin{equ}
\d_t u = \CL u + F_g^{(1)}(\d^*u)+ F_g^{(2)}(\d^*u)\noise\;,\qquad u_0 = \Phi_0\;.
\end{equ}
\end{assumption}

We will furthermore impose the following regularity 
assumption, where the various objects and
exponents are as introduced above.

\begin{assumption}\label{ass:diff}
The map $F$ is twice differentiable
as a map between the finite-dimensional vector spaces $\hat V_{<\gamma}$ 
and $\bar V_{<\bar \gamma}$. Its derivative $DF$ is such that
if $\PPhi$ and $J$ belong to $\CD^{\gamma,\eta}$ with values in $\hat V_{<\gamma}$
and $V_{<\gamma}$ respectively, then  the map
\begin{equ}
z\mapsto  \bigl(D F(\PPhi(z))\bigr)J(z)
\end{equ}
belongs to $\CD^{\bar \gamma,\bar \eta}$ and the map 
$(\PPhi,J) \mapsto D F(\PPhi)J$ is strongly locally Lipschitz between these spaces.
\end{assumption}
See \cite[Sec.~7.3]{reg} for the definition of strongly locally Lipschitz in this context.
We will also assume henceforth that we are given a continuous 
action $\tau$ of a suitable space of shifts onto the space $\MM$ 
of nice admissible models for $(\CT,\CG)$. This is formulated more precisely in the following 
structural assumption
on our equation, where $\zeta > 0$ denotes the regularity of the sector $V$ as
above.

\begin{assumption}\label{ass:shifts}
There exist $p \in (2,\infty)$ as well as a separable Banach space $X_0$ with
dense and continuous inclusions $\CC^\infty(\R^d) \subset X_0 \subset H_0$
such that
\begin{claim}
\item There is a continuous action $\tau$ of $\V = L^p([0,1],X_0)$ on $\MM$
such that, for all smooth functions $h, \omega$ and every $g \in \RR$, one
has the identity $\tau(h,M_g\LL(\noise)) = M_g \LL(\noise + h)$. 
\item One has $\V \subset \CC^{\gamma-2}$, extending functions on $[0,1]$ to vanish outside. 
\item The space $X_0$ admits $\CC^{\zeta}$ as a multiplier and, for every $\rho \in \CC_0^\infty(\R^d)$
integrating to $1$, the operators $h \mapsto \rho^\eps \star h$ are uniformly bounded
as $\eps \to 0$ and converge pointwise to the identity.
\end{claim}
Furthermore, there exists a $\CC^2$ map $G \colon \R \to \R$
such that, for any $h \in \V$ supported
in $(0,1)\times \R^d$ and any model 
$\PPi \in \MM$, the solution
$\PPhi^h$ to \eqref{e:SPDE} with model $\PPi^h = \tau(h,\PPi)$ is related to the solution
$\hPPhi^h$ to 
\begin{equ}[e:SPDEshift]
\hPPhi^h = \CP \one_+ \bigl(F(\hPPhi^h) + N \bigr) + P \star \bigl(G(\scal{\one,\hPPhi^h})h\bigr) + P \Phi_0\;,
\end{equ}
with model $\PPi$ by $\CR^h \PPhi^h = \CR \hPPhi^h$. 
Here, $\CR^h$ denotes the reconstruction operator for the model $\PPi^h$,
$\CR$ denotes the reconstruction operator for $\PPi$, and $\scal{\one,\cdot} \colon \hat V_{<\gamma} \to \R$ is 
the canonical projection  onto the component $\one$.
Finally, we assume that either $G$ is constant or $\hat V = V$.
\end{assumption}

\begin{remark}\label{rem:goodenough}
In \eqref{e:SPDEshift}, the term $G(\scal{\one,\hPPhi^h})h$ is interpreted as an element of 
$\V$, thanks to the assumption on $G$ and the fact that $\CC^{\zeta}$ is a multiplier on 
$X_0$. (The fact that $\scal{\one,\hPPhi^h}$ belongs to $\CC^\zeta$ follows from \cite[Prop.~3.28]{reg}.)
Since we assume that $h$ is supported away from $0$, the singularity at the origin
appearing in $G(\scal{\one,\hPPhi^h})$ plays no role. It then follows from the second assumption on
$X_0$ that one has $P \star \bigl(G(\scal{\one,\hPPhi^h})h\bigr) \in \CC^\gamma$, which 
we interpret as an element of $\CD^\gamma$
via its canonical lift to the Taylor polynomials $\bar \CT$. 
In particular, this shows that \eqref{e:SPDEshift} again satisfies the assumptions
of \cite[Sec.~7]{reg}, so that it does indeed admit unique local solutions. The identity
\eqref{e:SPDEshift} should then of course be interpreted as holding before the 
possible explosion time, and that this explosion time is the same for both equations.
\end{remark}

%
%
%

The framework considered here covers all current examples (on bounded domains
with periodic boundary conditions) of parabolic  
singular stochastic PDEs for which the theory of regularity structures
applies.
In the cases of the KPZ equation and the dynamical $\Phi^4_d$
equation with $d \in \{2,3\}$ as treated in \cite{KPZ,reg} one can take 
$N = \Xi$, where $\Xi \in \CT$ is the basis vector representing the driving noise. 
In the case of the one-dimensional stochastic heat equation with multiplicative 
noise covered in \cite{Etienne}, one can simply take $N = 0$.

The following proposition shows that this framework is compatible with our abstract result.

\begin{proposition}\label{prop:compatible}
In the setting of this section, Assumptions~\ref{ass:canonical}--\ref{ass:shifts} guarantee
that Assumptions~\ref{ass:cont}--\ref{ass:MallDer} 
of Section~\ref{sec:setup} are satisfied.
\end{proposition}

\begin{proof}
Recall that we have set $U = \CC^\eta$ and that, for $t$, $\Phi_0$ and $\PPi$ such that
the explosion time of \eqref{e:SPDE} is greater than $t$, one has
$\Phi_t(\Phi_0, \PPi) = \bigl(\CR \PPhi\bigr)(t,\cdot)$, where $\PPhi$ is the solution 
in $\CD^{\gamma,\eta}$ to \eqref{e:SPDE} and $\CR$ is the reconstruction operator associated
to the nice model $\PPi$. The only parts of Assumption~\ref{ass:cont} 
which do not follow immediately from the well-posedness results of \cite[Sec.~7]{reg}
are the continuity in time, since one may have $\eta < 0$, and the Fr\'echet differentiability with
respect to the initial condition. 
Time continuity follows from the fact that we assumed $\zeta > 0$, combined
with the condition $P\star \Pi_z N \in \CC(\R,\CC^\eta)$, which guarantee that 
$\CR \PPhi \in \CC(\R,\CC^\eta)$ for every $\PPhi \in \CD^{\gamma,\eta}(V)$.
To show Fr\'echet differentiability, we note that as a consequence of 
Assumption~\ref{ass:diff}, the map
\begin{equ}
(\tau,J) \mapsto \bigl(F(\tau), \bigl(D F(\tau)\bigr)J \bigr)\;,
\end{equ}
from $V_{<\gamma} \oplus V_{<\gamma}$ into $\bar V_{<\bar \gamma} \oplus \bar V_{<\bar \gamma}$
satisfies the assumptions of \cite[Thm~7.8]{reg}. As a consequence, we have unique
maximal solutions to the fixed point problem
\begin{equs}[e:DerSPDE]
\PPhi &= \CP \one_+ \bigl(F(\PPhi) + N\bigr) + P \Phi_0\;,\\
J &= \CP \one_+ \bigl(DF(\PPhi)J\bigr) + P J_0\;.
\end{equs}
As a consequence of the implicit function theorem and the fact that 
$\Phi \mapsto \CP \one_+ F(\PPhi)$ is a contraction on $\CD^{\gamma,\eta}([0,T],V)$
for small enough $T$, it follows that
the solution $\Phi_t$ is indeed differentiable in the initial condition and that
its derivative in the direction $J_0 \in \CC^\eta$ is given by $J$.
A standard patching argument shows that this is true not only for small $t$,
but all the way up to the explosion time of $\PPhi$.

We now turn to Assumption~\ref{ass:r}. If
$\Phi_t(u,\xi) = \death$, then we set $r_t(u,\xi) = +\infty$, otherwise,
we set
\begin{equ}
r_t(u,\xi) = \$\PPhi\$_{\gamma,\eta;[0,t]}\;,
\end{equ}
where $\PPhi$ is the solution to \eqref{e:SPDE} with initial condition $u$ and
underlying model $\xi$, and the norm $\$\cdot\$_{\gamma,\eta;[0,t]}$ is as
in \cite[Def.~6.2]{reg} with $\K = [0,t]\times \R^d$. (The fact that this set is
not compact is irrelevant since we only consider periodic functions.)
The $\CF_t$-measurability of $r_t$ is then an immediate consequence of the 
non-anticipativity of $P$, combined with the second part of Assumption~\ref{ass:model},
which guarantee that the random variable $\PPhi(s,x)$ is $\CF_s$-measurable
for all $s \in [0,1]$ and all $x \in \R^d$. The map is also increasing by 
definition and it is continuous in $t$, except at the explosion time when it has to
diverge to $+\infty$ as a consequence of the fact that the $\CC^\eta$-norm of the solution
is dominated by $r_t$ and has to blow up at the explosion time. The third required property of
$r$ is again satisfied by definition, while the last property follows from the fact that 
the fixed point map defining our solutions is locally Lipschitz continuous with values
in $\CD^{\gamma,\eta}$ as a function of both the initial condition and the underlying model.

To check that Assumption~\ref{ass:action} is verified, we exploit the identity
$\tau(h,M_g\LL(\noise)) = M_g \LL(\noise + h)$ given by Assumption~\ref{ass:shifts}.
Fix a mollifier $\rho$, then we have for every $\eps > 0$ and every $\omega \in \Omega$ 
the identity
\begin{equ}
\tau(h_\eps,M_{g_\eps}\LL(\omega_\eps)) = M_{g_\eps} \LL(\omega_\eps + h_\eps)\;,
\end{equ}
where we set $h_\eps = h\star \rho^\eps$ and similarly for $\omega_\eps$. For every 
bounded continuous $X\colon \CM \to \R$, we therefore have, by Girsanov's theorem,
\begin{equ}
\E X\bigl(\tau(h_\eps,M_{g_\eps}\LL(\omega_\eps))\bigr) = \E X\bigl(M_{g_\eps} \LL(\omega_\eps)\bigr) \exp(h_\eps^*(\omega_\eps) - \|h_\eps\|_\CH^2/2)\;.
\end{equ}
As a consequence of the third property of $X_0$, one has $h_\eps \to h$ in $\V$
and therefore also in $\CH$ since $X_0 \subset H_0$.
Furthermore $h_\eps^*(\omega_\eps) = \tilde h_\eps^*(\omega)$ with 
$\tilde h_\eps = h\star \rho^\eps\star \rho^\eps$, which also converges to 
$h$ in $\CH$, so that $h_\eps^*(\omega_\eps) \to h^*(\omega)$ in probability.
Taking limits on both sides and exploiting the fact that 
$M_{g_\eps}\LL(\omega_\eps) \to \mod$ in probability in $\CM$ and $\tau$ is jointly continuous,
the claim \eqref{e:definition} follows. The second claim of that assumption
immediately follows from Assumption~\ref{ass:repr}, in particular from the locality of the maps
$F_g^{(i)}$.

It remains to show that Assumption~\ref{ass:MallDer} holds, namely that the solution 
with model $\tau(h,\PPi)$ is Fr\'echet differentiable in the direction of $h \in \V$
at $h=0$. Thanks to Assumption~\ref{ass:shifts} however, we can fix the model $\PPi$ 
and it suffices to show that the solution $\hPPhi^h$ to \eqref{e:SPDEshift} is differentiable
at $h=0$ for every $t \in (0,1]$ and every $(\Phi_0,\PPi) \in \CN_t$. This however follows 
again immediately from the implicit functions theorem. 
\end{proof}

Of course we cannot expect that Assumption~\ref{ass:shiftmap} is satisfied without further condition
since we could for example encode a situation where $F_g^{(2)} = 0$, so that our solutions
are not random at all. The next subsection introduces a condition under which this last
assumption is also satisfied, thus yielding the strong Feller property for the corresponding
Markov process.

\subsection{A non-degeneracy assumption}

In order to have a chance for Assumption~\ref{ass:shiftmap} to hold, we need to know
that the Cameron-Martin space $\CH$ is sufficiently large to contain enough shifts
to be able to compensate for any shift in the initial condition. One possible
assumption guaranteeing that this is indeed the case is the following.

\begin{assumption}\label{ass:nondegenerate}
One has $\CC^\zeta \subset X_0 \subset  \CC^\eta$, and
the map $G$ appearing in Assumption~\ref{ass:shifts}
satisfies $G(u) > 0$ for every $u \in \hat V_{<\gamma}$.
\end{assumption}

This is then the final ingredient in the proof of the strong Feller property.

\begin{theorem}\label{theo:mainReg}
In the setting of this section, Assumptions~\ref{ass:canonical}--\ref{ass:nondegenerate} guarantee
that Assumptions~\ref{ass:cont}--\ref{ass:shiftmap} 
of Section~\ref{sec:setup} are satisfied. In particular, the Markov operator $\PP$
satisfies the strong Feller property.
\end{theorem}

\begin{proof}
We first note that, similarly to the proof of the first part of Proposition~\ref{prop:compatible},
our assumptions guarantee
that, for $(\Phi_0, \PPi) \in \CN_t$ and $0 \le s \le t$, the 
map $\Phi_{s,t}$ is Fr\'echet differentiable with respect to its first argument at the point
$\Phi_s(\Phi_0,\PPi)$. Denote its derivative in the direction $u$ by $J_{s,t}u$.
Since the second component of the
solution to the fixed point problem \eqref{e:DerSPDE} belongs to $\CD^{\gamma,\eta}(V)$
with $V \subset \bar \CT \oplus \CT_{\ge \zeta}$, it follows that for $t > s$, 
$J_{s,t}$ is a bounded linear operator from $\CC^\eta$ to $\CC^\zeta$.
(Since it is the identity for $s=t$, its norm blows up as $t \to s$ and one
can quantify this blow-up, but its details are not important to us.)

As a consequence of the variation of constants formula given in Corollary~\ref{cor:varConst}
of the appendix, the derivative of $\Phi_t$ with respect to $h \in \V$ is given by
\begin{equ}[e:MallDer]
\CD \Phi_t(\Phi_0,\PPi) h = \int_0^t J_{s,t} G(\scal{\one,\PPhi(s,\cdot)})h(s)\,ds\;,
\end{equ}
as long as $(\Phi_0,\PPi) \in \CN_t$.
Here, $\PPhi$ denotes the solution to \eqref{e:SPDE}, while $\Phi_t = \bigl(\CR \PPhi\bigr)(t,\cdot)$.

This suggests the following definition. First, fix an arbitrary bounded function
$\chi \colon [0,1] \to \R$ such that $\chi(s) = 0$ for $s \le 1/4$ and
such that $\int_0^1 \chi(s)\,ds = 1$.
For $s \le t$ and using the shorthand $J_r = J_{0,r}$ we then define $A_t^{(s)} \colon \CN_s \to L(U,\V_s)$ by
\begin{equ}[e:defAt]
\bigl(A_t^{(s)}(\Phi_0,\PPi) v\bigr)(r)  = {1\over t} \chi(r/t) G^{-1}(\scal{\one,\PPhi(r,\cdot)}) J_r v\;,\qquad r \in [0,s]\;.
\end{equ}
Here, the Jacobian $J_s$ does of course depend on $(\Phi_0,\PPi)$, but we have suppressed this dependency
in our notations. We also used the notation $G^{-1}$ for the map $u \mapsto 1/G(u)$.
Note that the right hand side depends on $s$ only through the range of the variable $r$.
In particular, one has the identity
\begin{equ}
\bigl(A_t^{(s)}(\Phi_0,\PPi) v\bigr)(r) = \bigl(A_t^{(\bar s)}(\Phi_0,\PPi) v\bigr)(r)\;,
\end{equ}
provided that $(\Phi_0,\PPi) \in \CN_s \cap \CN_{\bar s}$. We claim that this definition 
satisfies all the properties of Assumption~\ref{ass:shiftmap}. 
Indeed, $\CF_s$-measurability follows again from the fact that our solutions $\PPhi(s,\cdot)$ are
$\CF_s$-measurable. The identity \eqref{e:propertyShift} follows immediately from inserting
\eqref{e:defAt} into \eqref{e:MallDer} and using the fact that $J_{r,t}J_r = J_t$ for
every $r \in [0,t]$ and every $(\Phi_0,\PPi) \in \CN_t$. 
The fact that $A_t^{(s)}(\Phi_0,\PPi) v$ belongs to $\V$ for every $v \in U$ 
follows from the fact that $J_r$ maps $U$ into $\CC^\zeta$ as already mentioned
above, combined with the fact that $\chi$ is supported away from $0$ and 
the map $x\mapsto G(\scal{\one,\PPhi(r,x)})$ belongs to $\CC^\zeta$
by Assumption~\ref{ass:shifts}, so that the same is true for $x\mapsto 1/G(\scal{\one,\PPhi(r,x)})$ 
by the positivity of $G$.
It follows that one actually has the stronger statement that 
$A_t^{(s)}(\Phi_0,\PPi) v \in L^\infty([0,s],\CC^\zeta)$.
\end{proof}

\section{Application to concrete examples}

We now show that the abstract results obtained in this article can be 
applied to the concrete examples mentioned in the introduction. We will treat the case
of the dynamical $\Phi^4_3$ model in some detail and then only focus on the differences
with the other examples.

\subsection{Dynamical $\Phi^4_3$ model}
\label{sec:Phi43}

Recall that the dynamical $\Phi^4_3$ model is the stochastic PDE formally given by
\begin{equ}[e:Phi43]
\d_t \Phi = \Delta \Phi - \Phi^3 + \xi\;,
\end{equ}
where $\xi$ denotes space-time white noise, see \cite{reg}. 
The space variable is assumed to take values in some bounded three-dimensional torus.
This is of course only a formal notation, one way of constructing $\Phi$ is as limits
of solutions to the equation
\begin{equ}
\d_t \Phi_\eps = \Delta \Phi_\eps + C_\eps \Phi_\eps - \Phi_\eps^3 + \xi_\eps\;,
\end{equ}
where $\xi_\eps = \xi\star \rho^\eps$ for some compactly supported
space-time mollifier $\rho$ and for a suitable choice of (diverging) constants $C_\eps$.
Using regularity structures, it was shown in \cite{reg} that such limits exist
and are ``unique'' in the sense that modulo a suitable choice of $C_\eps$ they are independent
of the choice of mollifier $\rho$.

\begin{theorem}
The Markov semigroup generated by the solutions to the dynamical $\Phi^4_3$
model on $\CC^{-5/8}$ satisfies the strong Feller property.
\end{theorem}

\begin{proof}
In view of Theorems~\ref{theo:mainReg} and \ref{theo:main}, it suffices to 
show that Assumptions~\ref{ass:canonical}--\ref{ass:nondegenerate}
are satisfied. Recall the construction \cite[Secs.~9--10]{reg} of the regularity 
structure associated to \eqref{e:Phi43}. 
Consider first the space $\bar \CT$ of polynomials in four indeterminates $X_0,\ldots,X_3$,
representing the usual Taylor polynomials. For any multiindex $k = (k_0,\ldots,k_3)$,
we write $X^k$ for the monomial $X_0^{k_0}\cdots X_3^{k_3}$. We also write
$\one$ instead of $X^0$.
 
We then build a collection $\CU$ of formal expressions as the
smallest collection containing all the $X^k$ and $\CI(\Xi)$, and such that 
\begin{equ}[e:induction]
\tau_1,\tau_2,\tau_3 \in \CU \quad\Rightarrow\quad \CI(\tau_1\tau_2\tau_3) \in \CU\;,
\end{equ}
where it is understood that $\CI(X^k) = 0$ for every multiindex $k$.
Here, it is understood that we make the necessary identifications so that the 
product appearing on the right hand side is commutative and associative, with unit $\one$.
We then set 
\begin{equ}[e:defCW]
\CW = \{\Xi\} \cup \{\tau_1\tau_2\tau_3\,:\, \tau_i \in \CU\}\;,
\end{equ}
and we define our space $\CT$ as the set of all linear combinations of elements in $\CW$.
(Note that since $\one \in \CU$, one does in particular have $\CU \subset \CW$.)
Each formal expression in $\CW$ is assigned a degree by setting
$\deg \Xi = -{5\over 2} - \kappa$ for some (sufficiently small) $\kappa>0$, $\deg \one = 0$,
$\deg X_0 = 2$, $\deg X_i = 1$ for $i \in \{1,\ldots,3\}$, and then extending this
to all of $\CW$ by postulating that
\begin{equ}
\deg (\tau \bar \tau) = \deg \tau + \deg \bar \tau\;,\qquad
\deg \CI(\tau) = \deg\tau + 2\;.
\end{equ}
Naturally, $\CT_\alpha$ is then the subspace of $\CT$ spanned by
those elements in $\CW$ that are of degree $\alpha$. 
With this definition, provided that $\kappa$ is small enough ($\kappa < {1\over 2}$ suffices,
but one needs $\kappa$ to be even smaller later),
each space $\CT_{<\gamma}$ for $\gamma \in \R$ is finite-dimensional.
Using the same graphical notations as in \cite{reg,LDP} (bullets for instances of $\Xi$, lines
for instances of $\CI$, so that for example $\<22> = \CI(\Xi)^2\CI(\CI(\Xi)^2)$), $\CT_{<0}$
is for example spanned by
\begin{equ}[e:negsymb]
\{\Xi, \<1>\ , \<2>\ , \<3>\ , \<32>\ , \<22>\ , \<31>\}\;.
\end{equ}

The structure group $\CG$ built in \cite[Def.~8.20]{reg}
then consists of all linear maps $\Gamma\colon \CT \to \CT$
such that
\begin{claim}
\item One has $\Gamma \Xi = \Xi$ and there exist $x_i \in \R$ such that $\Gamma X_i = X_i - x_i \one$.
\item For every $\tau \in \CT_\alpha$ one has $\Gamma \tau - \tau \in \CT_{<\alpha}$.
\item For every $\tau$, $\bar \tau$ in $\CW$ such that $\tau \bar \tau\in \CW$,
one has $\Gamma(\tau\bar \tau) = (\Gamma\tau)(\Gamma \bar \tau)$.
\item For every $\tau \in \CW$ such that $\CI(\tau) \in \CW$, one has
$\Gamma \CI(\tau) - \CI(\Gamma \tau) \in \bar \CT$, where $\CI$ is extended from
$\CW$ to $\CT$ by linearity.
\end{claim}
These properties are consistent (in the sense that the last two properties never 
refer to any formal expression not contained in $\CW$) 
and they do form a group.

We then define $V \subset \CT$ as the subspace spanned by the $X^k$, as well
as all expressions of the form $\CI(\tau)$ for some $\tau$, \textit{except} for
the expression $\CI(\Xi)$. We do however set $N = \Xi$, so that $\hat V$ 
contains all $\CI(\tau)$, and we simply set $\bar V = \CT$. 
If $\kappa$ is sufficiently small, one can verify that these choices of $V$ and $\bar V$
satisfy the assumptions of our setting with exponents
\begin{equ}
\zeta = {1\over 2} - 3\kappa\;,\qquad \bar \zeta = -{3\over 2}-3\kappa\;.
\end{equ}
Possible choices for $\gamma$, $\bar \gamma$, $\eta$ and $\bar \eta$ 
(again provided that $\kappa$ is small enough) 
are given by 
\begin{equ}
\gamma = {7\over 4}\;,\qquad \bar \gamma = {1\over 4}\;,\qquad
\eta = -{5\over 8}\;,\qquad \bar \eta = -{15\over 8}\;.
\end{equ}
(The point for the $\eta$ exponents is that $\eta \in (-2/3,-1/2)$ and
$\bar \eta = 3\eta$ which are imposed by the irregularity  of the noise and the
form of the nonlinearity.)
These choices plainly satisfy the various inequalities and properties
we imposed in Section~\ref{sec:reg} with $q = 1$.
The nonlinearity $F\colon V_{<\gamma} \to \bar V_{<\bar \gamma}$ is simply given by
$F(\PPhi) = -\CQ_{<\bar \gamma} \PPhi^3$, where $\CQ_{<\bar \gamma}$ is the projection
onto terms of degree less than $\bar \gamma$. The fact that $F$ satisfies the strong 
local Lipschitz property was shown in \cite[Lem.~9.7]{reg}.
In order to verify Assumption~\ref{ass:canonical}, it thus only remains to exhibit $\LL$ and
$\RR$. The canonical lift $\LL$ is the same as in \cite[Sec.~8.2]{reg}. The group $\RR$ in 
this case is simply given by $\R^2$ with addition, acting on the space of admissible 
models as in \cite[Sec.~9.2]{reg}. It acts on $\CT$ in a natural way
by essentially performing the substitutions
$\<2> \mapsto C_1 \one$ and $\<22> \mapsto C_2 \one$ in the sense that 
an arbitrary symbol is mapped to the sum over all possible ways of contracting
(possibly multiple) occurrences of $\<2>$ and $\<22>$, multiplied by the corresponding 
powers of $C_i$. For example
\begin{equ}[e:act1]
M_g \<32> = \<32> + C_1\<30> + 3C_1\<12> + 3C_1^2\,\<10> + 3C_2\<1>\;.
\end{equ}
This action can be lifted to an action on the space $\MM$
of admissible models via the construction of \cite[Sec.~8]{reg}. (See also
\cite[]{BHZ} for a more systematic justification of this.)

The convergence of the mollified and renormalised models to a unique limiting model
$\mod$ is the content of \cite[Sec.~10.5]{reg}. The fact that the limiting model
is adapted to the filtration $\CF_t$ in the required way is shown in the same way
as in 
%
\cite{Etienne}, 
%
so that Assumption~\ref{ass:model} is also satisfied.
Assumption~\ref{ass:repr} is the content of \cite[Prop.~9.10]{reg}, which shows in particular that one
has
\begin{equ}
F_g^{(1)}(\d^* u) = c_g u - u^3\;,\quad
F_g^{(2)}(\d^* u) = 1\;,
\end{equ}
for a constant $c_g$ given by a suitable linear combination of the two components of $g$.
The regularity assumption in Assumption~\ref{ass:diff} follows immediately from the fact that
$F$ is polynomial. The fact that the derivative map $(\PPhi,J) \mapsto -3 \PPhi^2 J$ is strongly
locally Lipschitz between the spaces $\CD^{\gamma,\eta}$ and $\CD^{\bar \gamma, \bar \eta}$
follows from \cite[Prop.~6.12]{reg}. 

We now turn to Assumption~\ref{ass:shifts}, which is the one that requires
the largest amount of work. 
Regarding a suitable space
of shifts which generate a continuous action on our space $\CM$ of nice models,
we claim that the space $\CC^{-\kappa}(\R\times \R^d)$
has the required properties provided that $\kappa$ is sufficiently small,
so that we can the choose for $X_0$ space of the form $\CC^{\zeta}$
for $\zeta$ so that $\zeta \ge -\kappa$ and $X_0 \subset H_0$.
It is relatively easy to show this ``by hand'', but we want to have a 
more systematic proof which also carries over to the other examples. 
We introduce an auxiliary regularity structure
$(\hat \CT, \hat \CG)$ defined as follows, similarly to what 
was done in \cite[Sec.~3.2]{PeterMalliavin}.
The space $\hat \CT$ is constructed in exactly the same way
as $\CT$, with the exception that in the first step of the construction we start 
with a collection $\hat \CU$ that contains $X^k$ and $\CI(\Xi)$, as well as an additional
expression $\CI(\hat \Xi)$.
We then also add $\hat \Xi$ to the collection $\hat \CW$ otherwise defined from
$\hat \CU$ as above, and taken as a basis for the space $\hat \CT$. 
Graphically, if we denote $\hat \Xi$ by a circle, the symbols appearing in $\hat \CT$
are the same as those appearing in $\CT$, but with any occurrence of a bullet possibly
replaced by a circle, so for example $\CI(\Xi)\CI(\hat \Xi) = \<2h>$. 

Regarding the grading of $\hat \CT$, we set $\deg \hat \Xi = -\kappa$,
the degrees of the remaining basis vectors being obtained by using the same rules as above.
The structure group $\hat \CG$ is also defined as above by additionally
imposing that $\Gamma \hat \Xi = 0$ for every $\Gamma \in \hat \CG$. The same group $\RR$ 
acts naturally on 
$\hat \CT$ in exactly the same way as above. Denoting this action
by $\hat M$, one has for example,
for $g = (C_1,C_2)$
\begin{equ}
\hat M_g \<32h> = \<32h> + C_1\<30h> + C_1\<12h> + C_1^2\,\<10h> + C_2\<1h>\;.
\end{equ}
Compare this to \eqref{e:act1} where the constants are slightly different since there
are more inequivalent occurrences of $\<2>$ and $\<22>$ appearing in that symbol.
Thanks to \cite{BHZ} this can again be lifted naturally to an action on the corresponding
spaces of admissible models, so that we also have a space of ``nice models'' $\hat \CM$ for
this larger regularity structures.
Note also that one has $(\CT, \CG) \subset (\hat \CT, \hat \CG)$ in the sense
of inclusions of regularity structures as in \cite[Sec.~2.1]{reg}.

A crucial remark is that if $\kappa$ is sufficiently small (in our case one 
needs $\kappa > -1$ which guarantees that $\deg\<3h> > 0$), all of the elements
of $\hat \CW \setminus \CW$ are of strictly positive degree, except for $\hat \Xi$ 
itself. As a consequence, by repeatedly applying \cite[Prop.~3.31]{reg} and
\cite[Thm~5.14]{reg} and writing $\hat \MM$ for the space of admissible
models for $(\hat \CT, \hat \CG)$, there exists a \textit{unique} locally Lipschitz 
continuous map
$\CY\colon  \CC^{-\kappa} \times \MM \to \hat \MM$ such that
\begin{claim}
\item For every $h \in \CC^{-\kappa}$ and $\PPi \in \MM$, the model
$\hat \PPi^h = \CY(h,\PPi) \in \hat \MM$ agrees with $\PPi$ on $\CT \subset \hat \CT$
(or in other words $\hat \PPi^h$ extends $\PPi$) as in \cite[Def.~2.22]{reg}.
\item Writing $(\hat \Pi, \hat \Gamma) = \CY(h,\PPi)$, one has
$\hat \Pi_z \hat \Xi = h$ for every $z$.
\end{claim}
It is easy to show that the action of $\RR$ commutes with $\CY$ in the sense that, for every
$g \in \RR$, every $\PPi \in \CM$ and every $h \in \CC^{-\kappa}$, one has 
$\hat M_g \CY(h,\PPi) = \CY(h, M_g \PPi)$.
Indeed, both $\hat M_g \CY(h,\PPi)$ and $\CY(h, M_g \PPi)$
agree on $\CT \oplus \scal{\hat \Xi}$ and any admissible model is uniquely determined 
by this, as mentioned above.

It remains now to build a (locally Lipschitz continuous)
map $\CZ \colon \hat \MM \to \MM$, also commuting 
with the action of $\RR$, so that we can then define $\tau = \CZ \circ \CY$.
For this, given a model $\hat \PPi \in \hat \MM$, we introduce the following notion
of a ``$\hat\PPi$-polynomial'': 

\begin{definition}
A $\hat\PPi$-polynomial $f$ is a map $f \colon \R^{d+1} \to \hat\CT$
such that $f(z) = \Gamma_{z\bar z}f(\bar z)$ for any $z, \bar z$.
We say that $f$ is of degree $\deg f = \gamma$ if the component of $f(z)$ in
$\CT_{>\gamma}$ vanishes and its component in $\CT_{\gamma}$ is non-zero 
for some (and therefore all) $z$.
In particular, one has $f \in \CD^{\bar \gamma}$ for every $\bar \gamma > \deg f$.
\end{definition}

It is immediate that the product of two $\hat\PPi$-polynomials is again a
$\hat\PPi$-polynomial. 
Given an admissible model $\hat\PPi = (\Pi,\Gamma) \in \hat \MM$, we now define a collection
of $\hat\PPi$-polynomials $\{f_z^\tau\}$ for $\tau \in \CW$ and $z \in \R^{d+1}$
recursively as follows. 
Recall first that our model $\PPi$ defines operators 
$\CJ(z) \colon \hat\CT \to \bar \CT$ by
\cite[Eq.~5.11]{reg}. We also define an alternative notion $\hdeg$ of degree on
$\hat \CW$ by setting $\hdeg \tau = \deg \tau$ for $\tau \in \CW$,
but then setting $\hdeg \hat \Xi = \deg \Xi$ and defining it on the rest of $\hat \CW$
by using the same rules as for $\deg$. This allows us to define operators $\bar \CJ(z)\colon \hat \CT \to \bar \CT$ by setting
\begin{equ}[e:defJbar]
\bar \CJ(z) \tau = \CQ_{<\hdeg \tau} \CJ(z)\;.
\end{equ}
With these definitions at hand, we set
\begin{equ}
f_z^{X^k}(\bar z) = \Gamma_{\bar z z} X^k\;,\qquad 
f_z^{\Xi}(\bar z) = \Xi + \hat \Xi\;. 
\end{equ}
Then, we set recursively
\begin{equ}
f_z^{\tau \bar \tau}(\bar z)
= f_z^{\tau}(\bar z) f_z^{\bar \tau}(\bar z)\;,\qquad
f_z^{\CI(\tau)}(\bar z) = \bigl(\CI  + \CJ(\bar z) - \Gamma_{\bar zz} \bar \CJ(z)\Gamma_{z\bar z}\bigr)f_z^{\tau}(\bar z)\;.
\end{equ}
Using \cite[Lem.~5.16]{reg} it is a simple exercise to verify that 
if $f_z^\tau$ is a $\hat \PPi$-polynomial, then $f_z^{\CI(\tau)}$ as defined
above is indeed again a $\hat \PPi$-polynomial. Furthermore, it follows by induction
that, for every $z$ and every $\tau$, $f_z^{\tau}(\bar z)$ is a linear combination
of terms of $\hdeg$-degree equal to $\deg \tau$.
Combining this with the definition of $\bar \CJ(z)$ and the fact that 
$\deg\hat \Xi > \deg \Xi$ allows to show inductively that, for every $\tau \in \CW$, 
one has
\begin{equ}[e:propfzt]
\hat\CQ_{<\deg \tau} f_z^\tau(z) = 0\;,\qquad \forall z \in \R^d\;,
\end{equ}
where $\hat\CQ_{<\gamma}$ is the projection onto $\hat\CT_{<\gamma}$.

Define now operators $\Lambda_{z\bar z} \colon \CT \to \CT$
by setting
\begin{equ}
\Lambda_{z\bar z} \Xi = \Xi\;,\qquad \Lambda_{z\bar z} X^k = \Gamma_{z\bar z}X^k\;,
\end{equ}
and then recursively by
\begin{equ}
\Lambda_{z\bar z} (\tau \bar \tau) = (\Lambda_{z\bar z} \tau)(\Lambda_{z\bar z}\bar \tau)\;,
\end{equ}
as well as
\begin{equ}
\Lambda_{z\bar z} \CI \tau = \CI \Lambda_{z\bar z} \tau + \bigl(\bar \CJ(z) \Gamma_{z\bar z} - \Gamma_{z\bar z}\bar \CJ(\bar z)\bigr)f_{\bar z}^\tau(\bar z)\;.
\end{equ}
We also extend the definition of $f_z^\tau$ to all of $\tau \in \CT$ by linearity.
It is then a straightforward exercise to verify 
by recursion that one has the identity
\begin{equ}[e:algiden]
f_z^{\Lambda_{z\bar z} \tau} = f_{\bar z}^\tau\;.
\end{equ}
The map $\CZ$ is then defined as follows. Given a model $\hat \PPi = (\Pi,\Gamma) \in \hat\MM$,
we define a new model $\CZ \hat \PPi = (\tilde \Pi, \tilde \Gamma)$ by setting
\begin{equ}[e:defPitilde]
\tilde \Pi_z \tau = \CR f_z^\tau = \Pi_z f_z^\tau(z) \;,\qquad \tilde \Gamma_{z\bar z} = \Lambda_{z\bar z}\;,
\end{equ}
where $\CR$ is the reconstruction operator associated to $\hat \PPi$.
Let us verify that this is indeed an admissible model.
It follows from \eqref{e:algiden} and these definitions that one has 
the algebraic identity $\tilde \Pi_z \tilde \Gamma_{z\bar z} = \tilde \Pi_{\bar z}$ 
as required. Since the map $\tau \mapsto f_x^\tau$ is easily seen to be injective,
\eqref{e:algiden} also implies that $\Lambda_{z\bar z} \Lambda_{\bar z \underline z} = \Lambda_{z\underline z}$, so that it remains to verify that the required analytical bounds
hold. The bounds on $\tilde \Pi_z \tau$ follow at once from the corresponding bounds on
$\PPi$ and the fact that $\tilde \Pi_z \tau = \Pi_z f_z^\tau(z)$, combined with \eqref{e:propfzt}. 
The bounds on $\Lambda_{z\bar z}$ on the other hand follow inductively from the 
corresponding bounds on $\Gamma_{z\bar z}$, combined with \cite[Lem.~5.21]{reg}
and the fact that $f_z^\tau(z)$ is a linear combination of terms with 
$\hdeg$-degree equal to $\deg\tau$.
The fact that the new model is again admissible (in the sense that it realises $K$ for $\CI$) 
follows at once from the fact that since
\begin{equ}
\tilde \Pi_z \CI(\tau) =
\Pi_z \bigl(\CI  + \CJ(z) - \bar \CJ(z)\bigr)f_z^{\tau}(z)\;,
\end{equ}
and since 
\begin{equ}
\Pi_z \bigl(\CI  + \CJ(z)\bigr)\bar \tau = K \star \Pi_z \bar \tau
\end{equ}
for every $z$ and every $\bar \tau$ in the domain of $\CI$, it follows 
from the definitions \eqref{e:defJbar} and \eqref{e:defPitilde} 
of $\tilde \Pi_z$ and $\bar \CJ(z)$ that $\tilde \Pi_z \CI(\tau) - K \star \Pi_z \tau$
is a polynomial of degree $\tau$. Since furthermore it satisfies the desired analytical bounds,
it does indeed realise $K$ for $\CI$.

Setting $\tau(h,\PPi) = \CZ ( \CY(h,\PPi))$, we still need to show that $\tau$ is an 
action and that it commutes with the action of $\RR$.
Regarding the canonical lift $\LL$, it is immediate from the definitions that one has
\begin{equ}[e:taucanonical]
\tau(h,\LL(\xi)) = \LL(\xi + h)\;,
\end{equ}
for every smooth $\xi$ and $h$. If we can show that furthermore
\begin{equ}[e:wantedModel]
\tau(h,M_g \PPi) = M_g \tau(h,\PPi)\;,
\end{equ}
for every smooth $h$, every smooth $\PPi \in \MM$ and every $g \in \RR$, 
then both the required identity
for $\tau$ and the fact that it is an action follow by continuity.
Since, as before, admissible models on $\CT$ are uniquely determined by their
action on those elements $\tau \in \CW$ with $\deg\tau \le 0$, the identity
\eqref{e:wantedModel} only needs to be verified on the sector spanned by 
the elements in \eqref{e:negsymb}. This is non-trivial only for the last three
elements, but can easily be verified from the explicit formulae of both the 
action of $\RR$ and that of $\tau$. 

Let us also sketch a more systematic way of verifying \eqref{e:wantedModel}. For this, we 
note that the linear map $Z \colon \hat \CT \to \hat \CT$
obtained by substituting $\Xi$ by $\Xi + \hat \Xi$ in every formal expression of
$\hat \CW$ belongs to the ``renormalisation group'' of the regularity structure
$(\hat \CT, \hat \CG)$ in the sense of \cite[Def.~8.41]{reg} (see \cite[Appendix~B]{HQ} for a proof
of the fact that the second identity of \cite[Eq.~8.39]{reg} actually follows from the 
first one), thus yielding an action $\hat \CZ$ on $\hat \MM$. Furthermore,  
$\CZ$ is precisely given by $\hat \CZ$, followed by the canonical
projection from $\hat \MM$ to $\MM$. The claim \eqref{e:wantedModel} is then an
immediate consequence of the fact that the linear maps $Z$ and $M_g$ commute
on $\hat \CT$. This is indeed the case, with both $Z M_g$ and $M_g Z$ given by the map that 
maps a symbol $\tau$ to the sum over all ways of substituting instances of $\Xi$ by 
$\hat \Xi$, as well as contracting occurrences of $\<2>$ and $\<22>$ and replacing them 
by the corresponding renormalisation constants.
Since $L^p([0,1],\CC) \subset \CC^{-\kappa}$ for sufficiently large $p$,
we can for example choose $X_0 = \CC$, the space of 
$L$-periodic continuous functions. Furthermore, as a consequence of \eqref{e:taucanonical}
and \eqref{e:wantedModel}, it is immediate that \eqref{e:SPDEshift} holds with $G = 1$,
so that Assumption~\ref{ass:shifts} holds.
Finally, Assumption~\ref{ass:nondegenerate} is trivially verified since
$\CC^\zeta \subset \CC \subset L^2 \cap \CC^\eta$, and $G = 1$.
\end{proof}

\subsection{Multi-component KPZ equation}

Consider the system of coupled KPZ equations
\begin{equ}[e:cKPZ]
\d_t h^i = \d_x^2 h^i + S^{i}_{jk} \d_x h^j\,\d_x h^k + \xi_i\;,
\end{equ}
(summation over repeated indices is implied)
where the $\xi_i$ are independent space-time white noises on $\R \times \T^1$ 
and the $S^{i}_{jk}$ are constant coefficients. Such systems of equations
arise naturally when considering the large-scale limits of systems with more than
one locally conserved quantity, see \cite{MultiKPZ,SpohnFluct}.
We will sometimes assume that the coefficients $S$ satisfy the symmetry conditions
\begin{equ}[e:symmetry]
S^{i}_{jk} = 
S^{i}_{kj}\;,\qquad 
S^{i}_{jk} = 
S^{j}_{ki}\;. 
\end{equ}
As shown formally in \cite{MultiKPZ,SpohnFluct} and proven rigorously in \cite{Funaki}, 
this guarantees that the invariant measure
for \eqref{e:cKPZ} (modulo height shifts) is Gaussian and simply consist of independent
Brownian bridges. 
As usual, \eqref{e:cKPZ} should be interpreted as the limit, as $\eps \to 0$, of solutions to 
\begin{equ}[e:KPZeps]
\d_t h^i = \d_x^2 h^i + S^{i}_{jk} \d_x h^j\,\d_x h^k - C_i^{(\eps)} + \xi_i^{(\eps)}\;,
\end{equ}
where the constants $C_i^{(\eps)}$ are chosen in such a way that the bilinear term is given by
a Wick product with respect to the Gaussian structure determined by the linearised equation.
In the symmetric case \eqref{e:symmetry}, one chooses
\begin{equ}[e:constKPZ]
C_i^{(\eps)} = {c \over \eps} \sum_k S^i_{kk}\;,
\end{equ}
where $c$ is as in \cite[Eq.~1.4]{Hao}.
The symmetry condition \eqref{e:symmetry} is such that the additional 
logarithmically divergent renormalisation constants 
appearing for example in the analysis of \cite{KPZ} cancel out, otherwise
one may have to add a logarithmically diverging term to \eqref{e:constKPZ}.
As shown in \cite{KPZ,Peter} in the one-component case, but the multi-component case
does not add any difficulty whatsoever, the limit of \eqref{e:KPZeps} 
exists, up to the
first time at which the $\CC^\alpha$-norm of the solution blows up for some 
(and therefore all) $\alpha > 0$. It is now rather straightforward to verify that 
our assumptions are verified.

\begin{theorem}
The solutions to the coupled KPZ equations \eqref{e:KPZeps} generate a Markov semigroup
that satisfies the strong Feller property.
\end{theorem}

\begin{proof}
In this case, we simply set $N = 0$, so that $V = \hat V$. The construction of the 
corresponding regularity structure works in a way analogous to what we mentioned above
for the dynamical $\Phi^4_3$ model and is exposed for example in \cite{Peter}. 
The values of the exponents $\zeta$ and $\bar \zeta$ in this setting (assuming that 
the symbols $\Xi_i$ representing the noises $\xi_i$ have degree $-{3\over 2}-\kappa$) 
are given by  
\begin{equ}
\zeta = {1\over 2} - \kappa\;,\qquad \bar \zeta = -1-2\kappa\;,
\end{equ}
provided that $\kappa$ is sufficiently small.
Possible choices for $\gamma$, $\bar \gamma$, $\eta$ and $\bar \eta$ 
(again for $\kappa$ small enough) 
are given by 
\begin{equ}
\gamma = {7\over 4}\;,\qquad \bar \gamma = {1\over 4}\;,\qquad
\eta = {1\over 4}\;,\qquad \bar \eta = -{3\over 2}\;.
\end{equ}
(This time the main constraint on the $\eta$ exponents is $\eta \in (0,1/2)$ and
$\bar \eta = 2(\eta-1)$ which are imposed by the irregularity  of the noise and the
form of the nonlinearity.)
Again, these choices satisfy the various inequalities and properties
we imposed in Section~\ref{sec:reg} with $q = 1$.

The rest of the proof is virtually
identical to that given above for the dynamical $\Phi^4_3$ model.
\end{proof}

Denote now by $\tilde \CC^\alpha$ the space of $\CC^\alpha$ functions, quotiented by constant functions
and write $\mu$ for the probability measure on $\tilde \CC^\alpha$ (for some $\alpha \in (0,{1\over 2})$)
under which the
$h^i$ are independent Brownian bridges.
We then have the following result.

\begin{proposition}
Under condition \eqref{e:symmetry}, the measure $\mu$ is the unique invariant measure for
the multi-component KPZ equation and its solutions are almost surely global in time for
\textit{every} initial condition $h_0 \in \CC^\alpha$.
\end{proposition}

\begin{proof}
As already mentioned, the fact that $\mu$ is invariant for \eqref{e:cKPZ} was recently 
obtained in \cite{Funaki}.
The uniqueness of the invariant measure $\mu$ then follows immediately from Corollary~\ref{cor:unique}
since the Brownian bridge measure has full support in the space $\tilde \CC^\alpha$.
Since furthermore $\mu$ has exponential moments, the argument of \cite{Bourgain} 
(see also \cite{Konstantin} 
for an application in a context very similar to here) yields global solutions
for $\mu$-almost every initial conditions.

Let $A \subset \tilde \CC^\alpha$ be this set of full measure. Then $A$ is dense in $\tilde \CC^\alpha$
so that, for every $h \in \tilde \CC^\alpha$, there exists a sequence $h_n \in A$ with $h_n \to h$.
Using Theorem~\ref{theo:shift}, it follows that for every $\eps > 0$, 
there exist $t > 0$, $N > 0$ and a coupling
between the solutions with initial conditions $h$ and $h_N$ such that 
$\P((h,\mod(\omega)) \in \CN_t) > 1-\eps$ and, under this coupling, 
$\P(\Phi_s(h(s),\mod(\omega)) = \Phi_s(h_N(s),\mod(\bar \omega)) \,\forall s \ge t) > 1-\eps$.
Since the solution with initial condition $h_N$ exists for all times almost surely, it follows
that the solution  with initial condition $h$ also exists for all time with probability at least
$1-2\eps$. Since $\eps$ was arbitrary, the claim follows.
\end{proof}

\begin{remark}
The class of equations considered in \cite{CPAM} can be treated in exactly the same way, 
with the same exponents appearing.
In particular, using the fact that the measure $\mu$ in \cite[Eq.~4.2]{CPAM}
has full support, our result shows that it is the unique invariant measure for
the process constructed in \cite[Thm~4.1]{CPAM}.
\end{remark}

\begin{remark}
The class of equations formally given by 
\begin{equ}
\d_t u = \d_x^2 u + H(u) + G(u)\,\xi\;,
\end{equ}
with periodic boundary conditions
considered in \cite{Etienne} also satisfies the assumptions
of our theorem, provided that the function $G$ appearing in \cite[Eq.~1.2]{Etienne}
is strictly positive. (In the case when $G$ is a matrix, we need its singular values to stay 
away from $0$.)
This is the first example considered here where 
we need to consider a non-constant function $G$. 
In this particular example, the strong Feller property has long been
known, see \cite[Thm~7.1.1]{DPZ}, albeit under rather strong boundedness conditions
on the coefficients $G$ and $H$.

One example with non-constant $G$ to which our theory also applies
is the natural ``stochastic heat equation with values in a manifold'' 
considered in \cite{string,BHZ,Ajay}. In this case, the non-singularity 
of $G$ is a consequence of the fact that the Riemannian metric tensor of the
target manifold is strictly positive.
\end{remark}

\subsection{The dynamical $P(\Phi)_2$ model}

This is the model formally given by
\begin{equ}[e:Phi2]
\d_t \Phi = \Delta \Phi - P'(\Phi) + \xi\;,
\end{equ}
where $\xi$ denotes space-time white noise, $P$ is an even polynomial with positive 
leading coefficient, and
the space variable takes values in some bounded two-dimensional torus.
The strong Feller property for this model can be obtained in the same way as for the
dynamical $\Phi^4_3$ model, but the arguments are a bit easier since 
there is much more ``wriggle room''. Since the invariant measure for \eqref{e:Phi2}
(interpreted in a suitable Wick-renormalised sense) is known \cite{AlbRock91,DPD2}
and has full support, this again allows one to obtain almost sure global solutions
for every initial condition in $\CC^{-\kappa}$ for suitable $\kappa > 0$.
This result however has already been obtained by more PDE-oriented methods 
in \cite{Global,Pavlos} and the strong Feller property for \eqref{e:Phi2} has very recently been
obtained in \cite{Pavlos}, so we do not provide any more details here.

\appendix

\section{Variation of constants formula}

In this section, we derive a version of the variation of constants formula that is
suitable for our needs. This allows us to relate the derivative of the solution map
with respect to its initial condition to the derivative with respect to the driving noise
by \eqref{e:MallDer}, which was used in a crucial way in our proof. 
Throughout this appendix we assume that we are in the setting described in Section~\ref{sec:reg} and
that all of the assumptions we made there are satisfied, without further mentioning this
in our statements.

To formulate our result, the following notation is useful. Let $P_s$ be the hyperplane 
$\{(t,x) \in \R^{d+1} \,:\, t = s\}$ and write $\CD^{\gamma,\eta}_s \eqdef \CD^{\gamma,\eta}_{P_s}$ for the corresponding spaces with singularity at $P_s$
as defined in \cite[Def.~6.2]{reg}. Let also $W \subset \CT$ be a sector of regularity $\alpha_0$
of the regularity structure $(\CT,\CG)$.
Consider then a measurable (and $L$-periodic as usual) function
$F\colon [0,1]\times \R^{d+1} \to W_{<\gamma}$ with the property that 
$F_s(\cdot) \eqdef F(s,\cdot) \in \CD^{\gamma,\eta}_s$ for every $s \in [0,1]$ and write
\begin{equ}[e:defG]
\bar F(z) = \int_0^1 F_s(z)\,ds\;,
\end{equ}
for every $z \in \R^{d+1}$. We then have the following result.

\begin{lemma}\label{lem:integrate}
Let $\gamma > 0$ and $\eta \in (\gamma-2,\alpha_0)$, 
where $\alpha_0$ is the regularity of the sector $W$ above. Then, if 
$F_s$ is bounded in $\CD^{\gamma,\eta}_s$ uniformly over $s \in [0,1]$, the function
$\bar F$ given by \eqref{e:defG} belongs to $\CD^\gamma$.
\end{lemma}

\begin{proof}
For any pair of points $(z, \bar z)$ with $|z-\bar z| \le 1$ and any $\beta < \gamma$, we estimate 
$\|\bar F(z) - \bar F(\bar z)\|_\beta$ in the following way. Write $\delta = |z-\bar z|$ as a
shorthand and write $I_\delta = [t-4\delta^2, t+4\delta^2]$, where $t$ 
is such that $z= (t,x)$ for some $x \in \T^d$. We then 
write
\begin{equs}[e:boundG]
\|\bar F(z) - \Gamma_{z\bar z}\bar F(\bar z)\|_\beta &\le \int_{[0,1] \setminus I_\delta} \|F_s(z) - \Gamma_{z\bar z}F_s(\bar z)\|_\beta\,ds \\
&+ \int_{[0,1] \cap I_\delta} \bigl(\|F_s(z)\|_\beta  + \|\Gamma_{z\bar z} F_s(\bar z)\|_\beta\bigr)\,ds\;.
\end{equs}
For the first integral, we use the fact that, by the definition of the spaces $\CD^{\gamma,\eta}_s$,
one has the bound
\begin{equ}
\|F_s(z) - \Gamma_{z\bar z}F_s(\bar z)\|_\beta
\lesssim |z-\bar z|^{\gamma-\beta} |t-s|^{\eta-\gamma \over 2}\;.
\end{equ}
Since $\eta > \gamma - 2$ by assumption, this expression is integrable in $s$,
thus leading to the required bound of order $|z-\bar z|^{\gamma-\beta}$.
Regarding the second bound, it follows from the definition of the spaces $\CD^{\gamma,\eta}_s$
and the fact that we chose $\eta \le \alpha_0$ that
\begin{equ}
\|F_s(z)\|_\beta  + \|\Gamma_{z\bar z} F_s(\bar z)\|_\beta
\lesssim |t-s|^{{\eta-\beta \over 2}}
+ \sum_{\alpha \in [\beta,\gamma)} |\bar t-s|^{{\eta-\alpha \over 2}}|z-\bar z|^{\alpha-\beta}\;,
\end{equ}  
where $\bar t$ is the time component of $\bar z$. 
Since $\eta > \gamma - 2$, so that in particular $\eta > \beta - 2$, these functions
are all integrable in $s$ and we obtain the bounds
\begin{equ}
\int_{[0,1] \cap I_\delta} |t-s|^{{\eta-\beta \over 2}}\,ds \lesssim \delta^{\eta-\beta+2}\;,\quad
\int_{[0,1] \cap I_\delta} |\bar t-s|^{{\eta-\alpha \over 2}}\,ds \lesssim \delta^{\eta-\alpha+2}\;,
\end{equ}
so that the corresponding term in \eqref{e:boundG} is bounded by 
$|z-\bar z|^{\eta-\beta+2}$. Since furthermore $|z-\bar z| \le 1$
and $\eta + 2 > \gamma$, this is in turn bounded by $|z-\bar z|^{\gamma-\beta}$ as desired.
\end{proof}

We now consider linear equations of the type
\begin{equ}[e:linear]
J = \CP \one_+ \Psi J + P J_0\;,
\end{equ}
where, for some $T>0$, one has 
$\Psi \colon [0,T]\times \R^d \to L(V_{\gamma},\bar V_{\bar \gamma})$ and
we wrote $(\Psi J)(z) = \Psi(z)J(z)$. We make the following running assumption.

\begin{assumption}
The map $J \mapsto \Psi J$ with $(\Psi J)(z) = \Psi(z)J(z)$ 
maps $\CD^{\gamma,\eta}_s$ into $\CD^{\bar \gamma,\bar \eta}_s$
for every $s \ge 0$, with a bound of the type
\begin{equ}
\|\Psi J\|_{\bar \gamma,\bar \eta;[s,T]} \le C \|J\|_{\gamma,\eta;[s,T]}\;,
\end{equ}
holding uniformly over $s \in [0,T]$.
\end{assumption}

It follows immediately from \cite[Thm~7.8]{reg} that under this assumption, 
\eqref{e:linear} admits a unique global solution. Furthermore, by linearity of the equation,
this solution
is linear in the initial condition $J_0$. By our assumptions on the sector $V$, 
the solution $J$ is such that, away from $t=0$, $\CR J$ is a 
H\"older continuous function of regularity $\zeta$ by
\eqref{e:propSectors} and \cite[Prop.~3.28]{reg}.
We then write $J^{(s)} \colon \CC^\eta \to \CD^{\gamma,\eta}_s$ for the 
solution to
\eqref{e:linear} with $\one_+$ replaced by the indicator function of the
set $\{(t,x)\,:\, t \ge s\}$ and $P J_0$ replaced by the solution map
to the linearised equation with initial condition $J_s$, but starting at time $s$.
We also write $J_{s,t}\colon \CC^\eta \to \CC^\zeta$ for the linear maps such that,
for every $f \in \CC^\eta$, one has
\begin{equ}
\bigl(\CR J^{(s)}f\bigr)(t,x) = (J_{s,t} J_s f)(x)\;.
\end{equ}

It follows from \cite[Prop.~7.11]{reg} that these linear maps satisfy the identities
$J_{t,u}\circ J_{s,t} = J_{s,u}$ for any $0 \le s \le t \le u$, where we use of
course implicitly the canonical injection $\CC^\zeta \hookrightarrow \CC^\eta$.
Consider now a function $f \in L^p([0,1],\CC^\eta)$
and let $A$ be the solution to the fixed point equation
\begin{equ}[e:linearInhomogeneous]
A = \CP \one_+ \Psi A + P \star f\;.
\end{equ} 
As already mentioned in Remark~\ref{rem:goodenough}, our assumptions guarantee that
$P\star f \in \CC^\gamma$, which we interpret as an element of $\CD^\gamma$ by 
identifying it with its local Taylor expansion of order $\gamma$
at each point.
We then have the following version of the variation of constants formula.

\begin{proposition}\label{prop:varConst}
The equation \eqref{e:linearInhomogeneous} admits a unique local
solution in $\CD^\gamma$, given by the identity
\begin{equ}[e:varConst]
A = \int_0^\infty J^{(s)} f(s,\cdot)\,ds\;.
\end{equ}
\end{proposition}

\begin{proof}
Defining $A$ by \eqref{e:varConst}, we only need to show that 
\eqref{e:linearInhomogeneous} is satisfied. First, note that $A$ does indeed
belong to $\CD^\gamma$ as a consequence of Lemma~\ref{lem:integrate}.
We also note that, by the definition of $J^{(s)}$, one has
for every $s \ge 0$ the identity
\begin{equ}
J^{(s)} f(s,\cdot) = \CP \one_{t \ge s} \Psi J^{(s)}f(s,\cdot) + P_s f(s,\cdot)\;.
\end{equ}
Since furthermore $\bigl(J^{(s)} f(s,\cdot)\bigr)(t,\cdot) = 0$ for $t \le s$,
one also has
\begin{equ}
J^{(s)} f(s,\cdot) = \CP  \Psi J^{(s)}f(s,\cdot) + P_s f(s,\cdot)\;.
\end{equ}
Integrating over $s$ and exploiting the linearity of $\CP$, together with the fact that,
for $\bar \CT$-valued elements $f \in \CD^\alpha$, $\CP \one_+ f$ coincides with the
canonical lift of $P \star (\one_+ f)$, completes the proof.
\end{proof}

\begin{corollary}\label{cor:varConst}
For every $t \in [0,1]$ and every $(\Phi_0,\PPi) \in \CN_t$, the identity
\eqref{e:MallDer} holds. 
\end{corollary}

\begin{proof}
Applying the implicit functions theorem to \eqref{e:SPDEshift}, it follows that
\begin{equ}
\CD \Phi(\Phi_0,\PPi) h = \CR U^h\;,
\end{equ}
where $U^h \in \CD^{\gamma,\eta}$ solves the fixed point problem
\begin{equ}
U^h = \CP \one_+ \bigl(DF(\Phi) U^h \bigr) + P \star \bigl(G(\scal{\one,\Phi})h\bigr)\;,
\end{equ}
with $\Phi$ the solution to \eqref{e:SPDE} with initial condition $\Phi_0$.

On the other hand, we already argued that the derivative $J = D\Phi$ 
of the solution with respect its
initial condition $\Phi_0$ solves \eqref{e:DerSPDE}. We are therefore
almost exactly in the setting of Proposition~\ref{prop:varConst}.
The only problem is that the process $DF(\Phi)$ is only defined up to some 
possible blow-up time $T$, which is however guaranteed to satisfy $T > t$
by the definition of $\CN_t$. 
This can easily be circumvented by simply multiplying $DF(\Phi)$ with a smooth cutoff
function which leaves the equation unchanged before time $t$ and makes it vanish 
before time $T$, so that it can be continued trivially for all times.
We conclude that one does have the identity
\begin{equ}
U^h = \int_0^t J^{(s)} G(\scal{\one,\Phi(s,\cdot)})h(s) \,ds\;.
\end{equ}
Evaluating this at time $t$, the claim follows.
\end{proof}

\endappendix

\bibliographystyle{Martin}

\bibliography{refs}

\end{document}